\documentclass[12pt]{amsart}

\usepackage{enumitem}
\usepackage{verbatim}
\usepackage{xcolor}
\usepackage{hyperref}

\newcommand{\ph}{\varphi}
\newcommand{\eps}{\varepsilon}
\newcommand{\ulim}{\varlimsup}
\newcommand{\llim}{\varliminf}
\newcommand{\EE}{\mathbb{E}}

\newcommand{\NN}{\mathbb{N}}

\newcommand{\RR}{\mathbb{R}}

\newcommand{\FFF}{\mathcal{F}}
\newcommand{\EEE}{\mathcal{E}}

\newcommand{\SSS}{\mathcal{S}}

\newcommand{\GGG}{\mathcal{G}}

\newcommand{\WWW}{\mathcal{W}}
\newcommand{\bW}{\overline{\WWW}}

\newcommand{\htop}{h_{\mathrm{top}}}

\DeclareMathOperator{\Leb}{Leb}

\newcommand{\Es}{E^\mathrm{s}}
\newcommand{\Eu}{E^\mathrm{u}}
\newcommand{\Ws}{W^\mathrm{s}}
\newcommand{\Wu}{W^\mathrm{u}}

\newcommand{\Wcs}{W^\mathrm{cs}}
\newcommand{\Wcu}{W^\mathrm{cu}}

\newcommand{\ms}{m^\mathrm{s}}
\newcommand{\mmu}{m^\mathrm{u}}
\newcommand{\nnu}{\nu^\mathrm{u}}
\newcommand{\mcu}{m^\mathrm{cu}}
\newcommand{\mcs}{m^\mathrm{cs}}
\newcommand{\ns}{\nu^\mathrm{s}}
\newcommand{\ncs}{\nu^\mathrm{cs}}
\newcommand{\bmu}{\bar{m}^\mathrm{u}}
\newcommand{\bmcs}{\bar{m}^\mathrm{cs}}
\newcommand{\bmcu}{\bar{m}^\mathrm{cu}}
\newcommand{\bms}{\bar{m}^\mathrm{s}}

\newcommand{\phu}{\ph^\mathrm{u}}
\newcommand{\Phu}{\Phi^\mathrm{u}}

\newcommand{\Ru}{R^\mathrm{u}}
\newcommand{\Rcs}{R^\mathrm{cs}}

\DeclareMathOperator{\supp}{supp}
\DeclareMathOperator{\diam}{diam}
\DeclareMathOperator{\graph}{graph}

\newtheorem{theorem}{Theorem}[section]
\newtheorem{lemma}[theorem]{Lemma}
\newtheorem{proposition}[theorem]{Proposition}
\newtheorem{corollary}[theorem]{Corollary}

\newtheorem*{thma*}{Theorem}

\theoremstyle{remark}
\newtheorem{remark}[theorem]{Remark}

\theoremstyle{definition}
\newtheorem{definition}[theorem]{Definition}

\numberwithin{equation}{section}

\begin{document}

\title{SRB and equilibrium measures via dimension theory}
\author{Vaughn Climenhaga}
\address{Dept.\ of Mathematics, University of Houston, Houston, TX 77204}
\email{climenha@math.uh.edu}
\date{\today}

\thanks{The author was partially supported by NSF grant DMS-1554794.}

\begin{abstract}
It is well-known that SRB and equilibrium measures for uniformly hyperbolic flows admit a product structure in terms of measures on stable and unstable leaves with scaling properties given by the potential function. We describe a construction of these leaf measures analogous to the definition of Hausdorff measure, relying on the Pesin--Pitskel' description of topological pressure as a dimensional characteristic using Bowen balls. These leaf measures were constructed for discrete-time systems by the author, Ya.\ Pesin, and A.\ Zelerowicz.  In the continuous-time setting here, the description of the scaling properties is more complete, and we use a similar procedure with two-sided Bowen balls to directly produce the equilibrium measure itself. 
\end{abstract}

\maketitle

\begin{center}
\emph{For Tolya, who taught me to learn things a second time.}
\end{center}

\section{Introduction}

\subsection{Overview}

Our story takes place %in the part of dynamical systems lying at the confluence of 
where dynamical systems meets
two other major theories: thermodynamics and dimension theory. 
In dimension theory, Hausdorff dimension is defined via a family of measures with prescribed scaling behavior when we zoom in geometrically; for a given set, exactly one of these measures has the correct scaling properties. In thermodynamic formalism, one can use a similar approach to define topological entropy \cite{rB73} and topological pressure \cite{PP84} via a family of measures that have prescribed scaling behavior when we zoom in dynamically. Exactly one of these measures has the correct scaling properties.

The rest of this paper explores the following idea, introduced in \cite{CPZ,CPZ2}: \emph{for a uniformly hyperbolic system, the measure produced by this approach from dimension theory can be used to describe the equilibrium measure in thermodynamic formalism and its conditional measures on stable and unstable leaves}.

We emphasize that the leaf measures we study, and their relationship to the equilibrium measure, are by no means new; see \S\ref{sec:lit}. The novelty here consists of the following two aspects:
\begin{itemize}
\item the observation here and in \cite{CPZ,CPZ2} that these measures can be constructed using tools from dimension theory (though this was foreshadowed by \cite{uH89,bH89} which studied the measure of maximal entropy);
\item the adaptation of this technique to give a direct formula for the equilibrium measure itself, without working on the leaves.
\end{itemize} 
To illustrate this second point, we give a new construction for the Sinai--Ruelle--Bowen (SRB) measure on a hyperbolic attractor; see \S\ref{sec:def} for full definitions, and \S\ref{sec:two-sided} for an extension to equilibrium measures.
The construction uses a version of ``two-sided Bowen balls'' $B_{s,t}^*(x,r)$ whose definition requires
the following local product structure of a hyperbolic attractor $\Lambda$ for a smooth flow $(f_t)_{t\in\RR}$: there is a constant $L>0$ such that for all sufficiently small $r>0$, any $x\in \Lambda$, and any $z\in B(x,r)$, there is a point $y$ on the local unstable manifold of $z$ such that $f_\beta(y)$ is on the local stable manifold of $x$ for some $\beta = \beta(x,z) \in [-Lr,Lr]$. For large values of $s,t>0$, we will use the sets
\[
B^*_{s,t}(x,r) := \Big\{z\in \Lambda : \sup_{\tau \in [-s,t]} d(f_\tau z, f_\tau x) < r \text{ and } |\beta(x,z)| < \frac{r}{s+t} \Big\}.
\]

\begin{theorem}\label{thm:srb-main}
Let $M$ be a Riemannian manifold and $\Lambda\subset M$ a hyperbolic attractor for a smooth flow $(f_t)_{t\in\RR}$. 
Then for all sufficiently small $r>0$,
there is $c>0$ such that the SRB measure $\mu$ on $\Lambda$ is given by
\[
\mu(Z) = \lim_{T\to\infty} \inf \sum_{(x,s,t)\in \EEE}
\frac{c}{s+t} \det (Df_{s+t}|_{\Eu_{f_{-s}(x)}})^{-1}
\]
for all Borel $Z\subset \Lambda$,
where the infimum is over all $\EEE \subset \Lambda \times [T,\infty)^2$ such that $Z \subset \bigcup_{(x,s,t) \in \EEE} B^*_{s,t}(x,r)$.
\end{theorem}

\subsection{Topological pressure and equilibrium measures}

Throughout the paper, $M$ is a smooth compact Riemannian manifold, and $F = (f_t)_{t\in\RR}$ is a smooth flow on $M$. 
We will consider a compact $F$-invariant set $\Lambda\subset M$, which in all our results is assumed to be a topologically transitive locally maximal hyperbolic set. 
Given a continuous function $\ph\colon \Lambda\to \RR$, which we call a \emph{potential}, the \emph{topological pressure} $P(\ph)$ of $(\Lambda,F,\ph)$ can be defined in three ways. 

\begin{definition}[Variational principle]
\begin{equation}\label{eqn:var-princ}
P(\ph) = \sup \Big( h_\mu(f_1) + \int\ph\,d\mu\Big),
\end{equation}
where $h_\mu(f_1)$ is measure-theoretic entropy and the supremum is over all Borel $F$-invariant probability measures $\mu$ on $\Lambda$. 
\end{definition}

\begin{definition}[Growth rate]
Writing $\Phi(x,t) = \int_0^t \ph(f_s x)\,ds$,
\begin{equation}\label{eqn:growth}
P(\ph) = \lim_{r\to 0} \ulim_{t\to\infty} \frac 1t \log \inf \sum_{x\in E} e^{\Phi(x,t)},
\end{equation}
where the infimum is over all sets $E\subset \Lambda$ such that $\Lambda$ is covered by the collection $\{B_t(x,r) : x\in E\}$ of \emph{Bowen balls} given by
\[
B_t(x,r) = \{y\in M : d(f_s x, f_s y) < r \text{ for all }s\in [0,t]\}.
\]
\end{definition}

\begin{definition}[Dimension via a critical value]\label{def:dim}
Fixing $r>0$, define for each $\alpha\in\RR$ an outer measure $m^\alpha$ on $\Lambda$ by
\begin{equation}\label{eqn:ma}
m^\alpha(Z) = \lim_{T\to\infty} \inf \sum_{(x,t) \in \EEE} e^{\Phi(x,t) - t \alpha},
\end{equation}
where the infimum is over all finite or countable sets $\EEE \subset \Lambda\times [T,\infty)$ such that the Bowen balls $\{B_t(x,r) : (x,t) \in \mathcal{E}\}$ cover $Z$.\footnote{A crucial difference between this and the previous definition is that here the order $t$ may vary from one Bowen ball to the next.}
Then there is a critical value of $\alpha$ given by
\begin{equation}\label{eqn:crit}
P(\ph,r) = \sup \{\alpha : m^\alpha(\Lambda) = \infty \} = \inf \{\alpha : m^\alpha(\Lambda) = 0\},
\end{equation}
and the topological pressure is $P(\ph) = \lim_{r\to 0} P(\ph,r)$.
\end{definition}

A measure achieving the supremum in the variational principle is called an \emph{equilibrium measure}; this includes measures of maximal entropy and SRB measures. Questions of existence, uniqueness, and properties of equilibrium measures lie at the heart of the study of thermodynamic formalism for dynamical systems. %The first step is to construct an equilibrium measure, provided one exists. 

When $\Lambda$ is a topologically transitive locally maximal hyperbolic set and $\ph\colon \Lambda\to\RR$ is H\"older continuous, it is well-known that there is a unique equilibrium measure $\mu$, which is also the unique probability measure satisfying the following \emph{Gibbs property}: for all sufficiently small $r>0$ there is $Q=Q(r)$ such that\footnote{We write $A=Q^{\pm1}B$ to mean $Q^{-1}B \leq A \leq QB$.}
\begin{equation}\label{eqn:Gibbs}
%Q^{-1} \leq \frac{\mu(B_t(y,r))}{e^{-tP(\ph) + \int_0^t \ph(f_s y)\,ds}} \leq Q
\mu(B_t(x,r)) = Q^{\pm 1} e^{\Phi(x,t)-tP(\ph)}
\text{ for all } x\in\Lambda \text{ and } t>0.
\end{equation}
In this setting we have $P(\ph) = P(\ph,r)$ for all sufficiently small $r$, and choosing $\alpha = P(\ph) = P(\ph,r)$ we see that the weight in \eqref{eqn:ma} associated to each $(x,t)$ is exactly $e^{\Phi(x,t) - tP(\ph)}$ as in \eqref{eqn:Gibbs}.

In this paper we describe three related constructions of $\mu$ based on dimension theory, and in particular on the outer measure in \eqref{eqn:ma}.
\begin{enumerate}[itemsep=.5ex]
\item \emph{Push forward and average}. With $\alpha = P(\ph)$,
\eqref{eqn:ma} defines a finite Borel measure $\mmu_x$ on each local unstable manifold. The averaged pushforwards $\nu_t = \frac 1t \int_0^t (f_s)_* \mmu_x \,ds$ converge in the weak* topology to a scalar multiple of the unique equilibrium measure $\mu$. See Theorem \ref{thm:properties}.
\item \emph{Construct a product}. Reversing time gives an analogous construction of Borel measures $\ms_x$ on stable leaves.  Using the product structure (unstable) $\times$ (stable) $\times$ (flow direction) one can construct a product measure on neighborhoods in $\Lambda$ by $\mmu_x\times \ms_x\times\Leb$, which is equivalent to $\mu$.
%The corresponding density function can be written down explicitly, and thus $\mmu_x,\ms_x$ can be described as the conditional measures of $\mu$ multiplied by an explicitly given density function.
See Theorem \ref{thm:product}.
\item \emph{A two-sided definition}. Theorem \ref{thm:srb-main} is a specific case of a more general result that adapts Definition \ref{def:dim}, using two-sided Bowen balls to refine in both the stable and unstable directions, and refining ``manually'' in the flow direction, thus directly obtaining a Borel measure on $\Lambda$, which is a scalar multiple of $\mu$. See Theorem \ref{thm:direct}.
\end{enumerate}

In the first two constructions involving the leaf measures, an important role is played by the scaling properties of these measures under the flow and under holonomy maps; see Theorem \ref{thm:conf-cts}. In fact, these scaling properties characterize the leaf measures up to a scalar; see Corollaries \ref{cor:conditionals} and \ref{cor:unique}, where this is deduced from uniqueness of $\mu$ together with the fact that the density function in Theorem \ref{thm:product} can be written down explicitly, which allows us to recover the leaf measures from the conditional measures of $\mu$. In Corollary \ref{cor:srb-product}, we show how these results give a direct construction of the conditional measures of the SRB measure along \emph{stable} leaves.

\subsection{History and related results}\label{sec:lit}

The author is not aware of any analogues in the literature of the third construction mentioned above, via two-sided Bowen balls as in Theorem \ref{thm:srb-main} and \S\ref{sec:two-sided} below. For the first two constructions, the history is richer.

\subsubsection{Push forward and average}

Applying the ``push forward and average'' procedure to an initial reference measure is a standard way of proving the Krylov--Bogolubov theorem on existence of an invariant measure. For uniformly hyperbolic systems, Sinai \cite{jS68} and Ruelle \cite{dR76} proved that taking volume as the initial reference measure leads to a limiting measure whose conditionals on unstable leaves are equivalent to leaf volume, which in turn shows that this \emph{Sinai--Ruelle--Bowen (SRB) measure} governs the asymptotic behavior of volume-typical trajectories. Pesin and Sinai \cite{PS82} used leaf volume as the original reference measure, and generalized the result on unstable conditionals to the partially hyperbolic setting. See \cite{CLP17} for an overview of this approach and of SRB measures more generally.

This approach was extended beyond the SRB case to more general equilibrium measures in recent joint work of the author with Ya.\ Pesin and A.\ Zelerowicz \cite{CPZ,CPZ2}. This showed that for uniformly (and some partially) hyperbolic diffeomorphisms, \eqref{eqn:ma} defines a finite Borel measure on each unstable leaf whose averaged pushforwards converge to a multiple of the unique equilibrium measure. Theorem \ref{thm:properties} uses this to deduce the result for flows.

\subsubsection{Construct a product}

The ``construct a product'' procedure has been extensively explored for the measure of maximal entropy (MME), where it was carried out by Sinai for Anosov diffeomorphisms \cite{jS68} and Margulis for Anosov flows \cite{gM70}; this was extended to the Axiom A setting by Ruelle and Sullivan for diffeomorphisms \cite{RS75}, and Bowen and Marcus for flows \cite{BM77}. 
Margulis uses a functional analytic approach, while the other papers rely on Markov partitions.

Let us highlight two important precursors to the present work, which constructed the leaf measures for the MME as the Hausdorff measures associated to a dynamically defined metric. This was done for geodesic flows in negative curvature by 
Hamenst\"adt \cite{uH89}, and for more general Anosov flows by Hasselblatt \cite{bH89}. It is reasonable to view Theorem \ref{thm:product} as an extension of these results to equilibrium measures associated to nonconstant potential functions; see \S\ref{sec:alternate}.

For nonconstant potential functions, the product structure of the corresponding equilibrium measure was described by Haydn \cite{nH94} for flows, and by Leplaideur \cite{Lep} for diffeomorphisms. Both of these papers established the scaling properties of the leaf measures under the dynamics and under holonomy; Haydn referred to these properties as \emph{Margulis' cocycle equations}. Haydn's proof uses Markov partitions and the correspondence between equilibrium measures for a map and its suspension flow \cite{BR75}. Leplaideur's proof does not use a full-fledged Markov partition or symbolic dynamics, but does require a proper rectangle with the Markov property, on which the first return map is considered. Both proofs use the Perron--Frobenius operator and obtain the leaf measures via its eigendata.

In light of the previous paragraph, it is worth pointing out that the construction of an equilibrium state over a subshift of finite type using Ruelle's Perron--Frobenius operator \cite{Bow08} can be viewed as a product construction, even when the underlying system is non-invertible; one can interpret the eigenmeasure as a measure on local unstable leaves, and the eigenfunction as giving the total weight of local stable leaves in the natural extension.

\subsubsection{Reliance on the literature}\label{sec:reliance}

The proofs we give in \S\ref{sec:pf} involve a certain amount of ``cheating'' by relying on known results from the literature, up to and including the fact that there is a unique equilibrium measure and that this is also the unique invariant Gibbs probability measure.  With a little more effort (and a longer paper), the arguments here could be made self-contained. This would require us to provide proofs of the following facts.
\begin{enumerate}
\item Uniform counting bounds for leafwise partition sums: these play a key role in the proof that the measures we construct are positive and finite. See \cite[\S6]{CPZ2} for details of this argument in the discrete-time case, using ideas that go back to Bowen's work using the specification property \cite{rB745}.
\item Ergodicity of the invariant measure built using any of the three constructions. In \cite{CPZ,CPZ2} this is done for the discrete-time case using the Hopf argument.
\item An ergodic Gibbs measure must be the unique equilibrium measure; this argument is standard and can be found in \cite{rB745}.
\end{enumerate}
We have chosen to shorten the proofs in this paper by relying on known results, but for potential applications to non-uniform hyperbolicity (see the next section) it would likely be productive to seek direct proofs of these facts along the lines in \cite{CPZ,CPZ2}.

\subsection{Non-uniform hyperbolicity}

For uniformly hyperbolic systems, thermodynamic formalism is very well understood, and in this setting the results in this paper are mostly intended to illuminate a well-known subject from a slightly different angle. It may be hoped, however, that the approach presented here will eventually bear fruit for non-uniformly hyperbolic systems, where the story is much less complete \cite{CP17}.

One well-established approach to thermodynamic formalism for non-uniformly hyperbolic systems uses a Markov structure in one guise or another, via countable-state Markov partitions, Young towers, or inducing schemes; see \cite{CP17} for an overview of the relevant literature. More recently, the author and D.J.\ Thompson have introduced an alternate approach based on Bowen's proof of uniqueness for uniformly hyperbolic systems using expansivity and specification \cite{rB745}. There are non-uniform versions of these properties \cite{CT12,CT16} that have found applications in examples such as geodesic flow in nonpositive curvature \cite{BCFT} and no conjugate points \cite{CKW}; see the forthcoming survey \cite{CT21} for a more complete description of this approach.

If one wants to go beyond uniqueness of the equilibrium measure and establish strong statistical properties, local product structure, and so on, then the well-established symbolic approach gives stronger results than the more recent specification-based approach.\footnote{However, it is worth mentioning recent work of Call and Thompson \cite{CaT19,Cal20} that uses the specification approach to establish the K property, and even Bernoullicity in some cases.} Given that the approach in the present paper shares much in common with the specification-based approach, such as a reliance on uniform counting bounds and bare-hands constructions, it is worth asking whether the present approach can also be generalized to non-uniform hyperbolicity, and then used to establish some of these stronger results.

One example of this symbiosis is found in recent work by the author, Gerhard Knieper, and Khadim War \cite{CKW} that used the specification approach to establish uniqueness of the MME for geodesic flows on surfaces with no conjugate points, and then used the Patterson--Sullivan approach to obtain a product structure for this measure; this in turn can be used to establish the Margulis asymptotics for the number of closed geodesics on such surfaces \cite{CKW2}. The Patterson--Sullivan approach can be viewed as a leafwise construction analogous to the present one \cite{vK90}, and thus one may reasonably hope to carry out a similar procedure in other non-geometric examples, including more general non-uniformly hyperbolic systems, hyperbolic systems with singularities such as billiards, or singular hyperbolic flows such as the Lorenz attractor. Of course one would need to first carry out the specification approach in these settings, or at least adapt enough of it to prove the uniform counting bounds necessary for the construction here, and so for the time being this is all rather aspirational.

\subsection*{Outline of the paper}

In \S\ref{sec:def}, we recall some background definitions from the literature on hyperbolic sets for flows and on Carath\'eodory dimension characteristics. In \S\ref{sec:results} we state the main results that were briefly described above. We give the proofs in \S\ref{sec:pf}.

\section{Background definitions}\label{sec:def}

The reader who is familiar with hyperbolic dynamics can likely skip this section on a first reading and proceed directly to the main results in \S\ref{sec:results}.
This section contains basic definitions and standard facts about hyperbolic sets for flows (\S\ref{sec:hyp-sets}), including some lemmas that must be stated before our main results in order to establish the scale at which various constructions are made; it also describes the general notion of Carath\'eodory dimension characteristic (\S\ref{sec:car-dim}) that we use when we define various measures in \S\ref{sec:results}. We refer to the recent book of Fisher and Hasselblatt \cite[Chapters 5 and 6]{FH19} for proofs of facts about hyperbolic flows that are omitted here, and to Pesin's book \cite{pes97} for a more in-depth treatment of dimension theory in dynamical systems.

%Before defining the leaf measures and stating our main results in \S\ref{sec:results}, we recall in \S\S\ref{sec:hyp-sets}--\ref{sec:basic-sets} some basic definitions and well-known facts about hyperbolic sets for flows, as well as some basic lemmas that must be stated before our main results in order to establish the scale at which various constructions are made. For further details of the general definitions and background proofs, we refer to the recent book of Fisher and Hasselblatt \cite{FH19}. In \S\ref{sec:car-dim} we recall the general notion of Carath\'eodory dimension characteristic, referring to Pesin's book \cite{pes97} for a more in-depth treatment.

\subsection{Hyperbolic sets}\label{sec:hyp-sets}

Throughout, recall that $M$ is a smooth compact Riemannian manifold, and $F = (f_t)_{t\in\RR}$ is a smooth flow on $M$. 

\begin{definition}[Hyperbolic sets]\label{def:hyp-set}
A \emph{hyperbolic set} for $F$ is a compact $F$-invariant set $\Lambda \subset M$ over which the tangent bundle admits a $DF$-invariant splitting $T_\Lambda M = \Es \oplus \Eu \oplus E^0$ such that $E^0$ is the flow direction,\footnote{We do not allow $\Lambda$ to contain fixed points for the flow, so $E^0$ is always a one-dimensional subspace.} $\Es$ is uniformly contracting, and $\Eu$ is uniformly expanding: there are $C_0\geq 1$ and $\chi \in (0,1)$ such that
\begin{equation}\label{eqn:hyp}
\|Df_t|_{\Es}\| \leq C_0 \chi^t
\quad\text{and}\quad
\|Df_{-t}|_{\Eu}\| \leq C_0 \chi^t
\quad\text{for all }t\geq 0.
\end{equation}
\end{definition}

By passing to an \emph{adapted metric} and increasing $\chi$ slightly,\footnote{If we do not pass to an adapted metric, then our constructions give equivalent measures, but we do not have a formula for the Radon--Nikodym derivatives and thus cannot write down the product construction explicitly, or recover the leaf measures from the conditional measures; see \S\ref{sec:alternate}. In \S\ref{sec:two-sided} we will drop the assumption that the metric is adapted.}
we assume without loss of generality that $C_0=1$. In the adapted metric, we can also assume that $\Es,\Eu,E^0$ are all orthogonal, and that the flow has unit speed so that $\|Df_t|_{E^0}\|=1$ for all $t\in\RR$.

We write $d$ for the distance function on $M$ induced by the (adapted) Riemannian metric, and $B(x,r)$ for the ball centered at $x\in M$ with radius $r>0$. More generally, given an injectively immersed connected submanifold $W\subset M$, we will write $d_W$ for the distance function on $W$ induced by the restriction of the Riemannian metric, and 
\begin{equation}\label{eqn:BW}
B(x,r,W) = \{y\in W : d_W(y,x) < r\}
\end{equation}
for the corresponding ball in $W$ centered at $x\in W$ with radius $r>0$.  

\begin{definition}[Cone fields on $\Lambda$]\label{def:cone}
Given $\kappa>0$, the \emph{stable and unstable cone fields of width $\kappa$ on $\Lambda$} are
\begin{equation}\label{eqn:cones}
\begin{aligned}
C_x^u = C_x^u(\kappa) &:= \{ v + w : v\in \Eu_x,\ w \in \Es_x \oplus E^0_x,\ \|w\| < \kappa \|v\| \}, \\
C_x^s = C_x^s(\kappa) &:= \{ v + w : v\in \Es_x,\ w \in \Eu_x \oplus E^0_x,\ \|w\| < \kappa \|v\| \}.
\end{aligned}
\end{equation}
\end{definition}

For every $t> 0$ and $x\in \Lambda$ we have
\begin{equation}\label{eqn:cone-inv}
\overline{Df_t(x)(C_x^u)} \subset C_{f_t x}^u
\quad\text{and}\quad
\overline{Df_{-t}(x)(C_x^s)} \subset C_{f_{-t} x}^s.
\end{equation}
Moreover, for every $\theta,\kappa>0$ there is $t_0\in\RR$ such that given any $x\in \Lambda$ and $v\in T_xM$ with $\angle(v,\Es_x \oplus E^0_x)\geq\theta$, we have
\begin{equation}\label{eqn:push-cone}
Df_t(x)(v) \in C_{f_t x}^u(\kappa)
\quad\text{for all }t\geq t_0.
\end{equation}
Fixing $\bar\chi \in (\chi,1)$, by \eqref{eqn:hyp} and \eqref{eqn:cone-inv} we can choose $\kappa>0$ sufficiently small that for all $x\in \Lambda$ and $t\geq 0$, we have
\begin{equation}\label{eqn:cone-hyp}
\begin{aligned}
\|Df_t(x)(v)\| \geq \bar\chi^{-t}\|v\|
&\text{ for all } v\in C_x^u(\kappa), \\
\|Df_{-t}(x)(v)\| \geq \bar\chi^{-t}\|v\|
&\text{ for all } v\in C_x^s(\kappa).
\end{aligned}
\end{equation}
From now on, $C_x^u$ and $C_x^s$ will be defined with this value of $\kappa$ unless otherwise specified.

\begin{definition}[Extension to a neighborhood]\label{def:nbhd}
Extend the distributions $\Es$ and $\Eu$ continuously to a neighborhood of $\Lambda$ (note that they need not be invariant off of $\Lambda$). Define cone families $C_x^s,C_x^u$ on this neighborhood by \eqref{eqn:cones}, with $\kappa$ as in \eqref{eqn:cone-hyp}. Fixing $\lambda \in (\bar\chi,1)$, there is $r_0>0$ such that if $x\in M$ and $T>0$ have the property that $f_t(x) \in B(\Lambda,r_0) := \bigcup_{y\in \Lambda} B(y,r_0)$ for every $t\in [0,T]$, then for all $t\in (0,T]$ we have
\begin{equation}\label{eqn:nearby-cone-hyp}
\overline{Df_t(x)(C_x^u)} \subset C_{f_t x}^u
\text{ and }
\|Df_t(x)(v) \| \geq \lambda^{-t} \|v\| \text{ for all } v\in C_x^u,
\end{equation}
and similarly for the stable cones with time reversed.
\end{definition}

\begin{definition}[Admissible manifolds]\label{def:admissibles}
An embedded connected submanifold $W \subset B(\Lambda,r_0)$ is \emph{$u$-admissible} 
if $T_xW \subset C_x^u$ for all $x\in W$,\footnote{This definition does not require $W$ to have the same dimension as $\Eu$, but in practice we will only consider $u$-admissible manifolds of maximal dimension.}
and \emph{$s$-admissible} if $T_xW\subset C_x^s$ for all $x\in W$.
\end{definition}

It follows from \eqref{eqn:nearby-cone-hyp} that
\begin{multline}\label{eqn:u-admissible}
\text{if $W$ is $u$-admissible and $f_\tau(W) \subset B(\Lambda,r_0)$ for all $\tau \in [0,t]$,}\\
\text{then $d_W(x,y) \leq \lambda^t d_{f_t W}(f_t x, f_t y)$ for all $x,y\in W$.}
\end{multline}
A similar result holds for $s$-admissible manifolds if we replace $f_\tau$ and $f_t$ by $f_{-\tau}$ and $f_{-t}$. 

\begin{definition}[Global manifolds $\Ws(x),\Wu(x)$]
Given a point $x\in \Lambda$, the \emph{global stable and unstable manifolds} of $x$ are
\begin{align*}
\Ws(x) &= \{y\in M : d(f_t y, f_t x) \to 0 \text{ as } t\to\infty\}, \\
\Wu(x) &= \{y\in M : d(f_t y, f_t x) \to 0 \text{ as } t\to-\infty\}.
\end{align*}
These are injectively immersed manifolds such that $T_x W^*(x) = E^*_x$ for every $x\in \Lambda$ and $*\in\{\mathrm{s,u}\}$. They are clearly flow-invariant in the sense that $f_t(W^*(x)) = W^*(f_t x)$ for every $x\in \Lambda$, $*\in\{\mathrm{s,u}\}$, and $t\in \RR$.
\end{definition}

\begin{definition}[Local manifolds $\Ws(x,\delta),\Wu(x,\delta)$]\label{def:local-leaves}
Given $x\in \Lambda$ and $\delta>0$, let $\Ws(x,\delta) := B(x,\delta,\Ws(x))$, and similarly for $\Wu$. There exists $\delta_0>0$ such that for all $x\in \Lambda$ and $\delta \in (0,\delta_0]$, the set $W = \Ws(x,\delta)$ has the following properties:
\begin{itemize}
\item it is $s$-admissible;\footnote{Observe that in the Anosov case $\Lambda = M$, every submanifold of $\Ws(x)$ is $s$-admissible, but that when $\Lambda \neq M$ there may be submanifolds of $\Ws(x)$ thare are not $s$-admissible because they contain points that are too far from $\Lambda$.}
\item there is a topological ball $B\subset \Es_x$ containing the origin and a function $\psi\colon B\to \Eu_x$ such that $W = \exp_x(\graph\psi)$. 
%is the image under $\exp_x\colon T_x M \to M$ of the graph of $\psi$.
\end{itemize}
We refer to any $W \subset \Ws(x)$ with these properties as a \emph{local stable manifold}, or a \emph{local stable leaf}. Local unstable manifolds are defined analogously, reversing the roles of $s,u$.
These local leaves are $C^1$ and depend continuously on $x$ in the $C^1$ topology.
\end{definition}

If $W$ is a local stable manifold of $x$, then $f_t(W)$ is a local stable manifold of $f_t(x)$ for all $t\geq 0$, and similarly for unstable manifolds with $t\leq 0$.
Thus \eqref{eqn:u-admissible} and its analogue for stable manifolds gives
\begin{equation}\label{eqn:leaves-contract}
\begin{aligned}
d(f_t y, f_t z) &\leq \lambda^t d(y,z) \text{ for all } y,z\in \Ws(x,\delta_0) \text{ and } t\geq 0, \\
d(f_{-t} y, f_{-t} z) &\leq \lambda^t d(y,z) \text{ for all } y,z\in \Wu(x,\delta_0) \text{ and } t\geq 0,
\end{aligned}
\end{equation}
where here and throughout the remainder of the paper $\delta_0 >0$ is fixed as in Definition \ref{def:local-leaves}.

\begin{definition}[Weak manifolds $W^{cs,cu}$]
Given $x\in \Lambda$ and $\delta\in (0,\delta_0]$, we consider the following \emph{local weak stable and unstable manifolds}:
\begin{align*}
\Wcs(x,\delta) &= \{ f_t(y) : y\in \Ws(x,\delta), |t| < \delta \}, \\
\Wcu(x,\delta) &= \{ f_t(y) : y\in \Wu(x,\delta), |t| < \delta \}.
\end{align*}
We also consider the \emph{global weak stable and unstable manifolds}
\begin{align*}
\Wcs(x) &= \{ f_t y : y\in \Ws(x), t\in \RR \} = \bigcup_{t\in \RR} \Ws(f_t x), \\
\Wcu(x) &= \{ f_t y : y\in \Wu(x), t\in \RR \} = \bigcup_{t\in \RR} \Wu(f_t x).
\end{align*}
\end{definition}

\begin{definition}[Weak-stable transversals]\label{def:cs-trans}
A $C^1$ injectively immersed connected submanifold $W\subset M$ is \emph{transverse to the weak-stable direction} if $T_x M = T_x W \oplus T_x \Wcs(x) = T_xW \oplus \Es_x \oplus E^0_x$ for all $x\in W \cap \Lambda$. Write $\WWW^u$ for the set of all such $W$.

Given $\theta>0$, let $\WWW^u_\theta$ denote the set of $W\in \WWW^u$ such that $\angle(w,v) \geq \theta$ for any $x\in W\cap \Lambda$, $w\in T_x W$, and $v\in T_x\Wcs(x) = \Es_x \oplus E^0_x$.
Reversing the roles of $u$ and $s$, we define $\WWW^s$ and $\WWW^s_\theta$ analogously.
\end{definition}

Using \eqref{eqn:push-cone}, there exists $r_1\in (0,r_0]$ such that the following is true.

\begin{lemma}\label{lem:push-trans}
Given any $W\in \WWW^u$ and $x\in W\cap \Lambda$, we have
\begin{equation}\label{eqn:push-to-Wu}
\lim_{t\to\infty} d_{C^1}(B(f_t x, r_1, f_t W), \Wu(f_t x, r_1)) = 0.
\end{equation}
Given $\theta>0$, this convergence is uniform over all $W\in \WWW^u_\theta$ and $x\in W\cap \Lambda$.
In particular, there is $t_1=t_1(\theta)\geq 0$ such that for any $W \in \WWW^u_\theta$, $x\in W \cap \Lambda$, and $t\geq t_1$, the manifold $B(f_t x, r_1, f_t W)$ is $u$-admissible.
\end{lemma}

Our main results will be given for a hyperbolic set satisfying the following two additional properties.
\begin{itemize}
\item \emph{Local maximality}: $\Lambda = \bigcap_{t\in\RR} f_t(U)$ for some open $U\supset \Lambda$.
\item \emph{Topological transitivity}: $\Lambda$ contains a dense orbit.
\end{itemize}

\begin{definition}[Basic sets]\label{def:basic-set}
A topologically transitive locally maximal hyperbolic set $\Lambda$ will be referred to as a \emph{basic set}.
\end{definition}

\begin{definition}[Local product structure]\label{def:lps}
There are $L,\delta_1>0$ such that if $x,y\in \Lambda$ have $d(x,y) < \delta_1$, then the intersection $\Wcs(x,L\delta_1) \cap \Wu(y,L\delta_1)$ consists of a single point, which we denote $[x,y]$ and refer to as the \emph{bracket} of $x$ and $y$.\footnote{This is sometimes called the \emph{Bowen bracket} or the \emph{Smale bracket}.} Moreover, $[x,y]$ lies in both of $\Wcs(x,Ld(x,y))$ and $\Wu(y,Ld(x,y))$.
A hyperbolic set $\Lambda$ has \emph{local product structure} if $[x,y]\in \Lambda$ for every $x,y\in \Lambda$ with $d(x,y)<\delta_1$.
\end{definition}

Given a hyperbolic set $\Lambda$, local product structure is equivalent to local maximality. In particular, basic sets have local product structure.

\begin{definition}[Attractors]\label{def:attractors}
A basic set $\Lambda$ is an \emph{attractor} if there is an open set $U\supset \Lambda$ such that $\overline{f_t(U)} \subset U$ for all $t>0$.
\end{definition}

\subsection{Carath\'eodory dimension and measure}\label{sec:car-dim}

We recall a construction from \cite[\S10]{pes97} generalizing Hausdorff dimension and measure.

\begin{definition}
A \emph{Carath\'eodory dimension structure}, or \emph{C-structure}, on a set $X$ is given by the following data.
\begin{enumerate}
\item An indexed collection $\FFF = \{U_s \subset X : s\in \SSS\}$ of subsets of $X$.
\item Functions $\xi,\eta,\psi \colon \SSS\to [0,\infty)$ such that
\begin{enumerate}[label=(H\arabic{*})]
\item %if $U_s = \emptyset$, then $\eta(s) = \psi(s) = 0$; if $U_s \neq\emptyset$, then $\eta(s) > 0$ and $\psi(s)>0$;
$\eta(s) = 0$ if and only if $U_s = \emptyset$, and similarly for $\psi$;\footnote{In \cite{pes97} there is an extra requirement that the case $U_s=\emptyset$ occur, but this can be safely omitted by defining $m^\alpha_C(\emptyset)=0$ in \eqref{eqn:mCa}.}
\item for every $\delta>0$, there exists $\eps>0$ such that $\eta(s) \leq \delta$ for any $s\in \SSS$ with $\psi(s) \leq \eps$;
\item for every $\eps>0$, there is a finite or countable $\GGG \subset \SSS$ such that $\bigcup_{s\in\GGG} U_s \supset X$ and $\psi(\GGG) := \sup \{ \psi(s) : s\in \SSS \} \leq \eps$.
\end{enumerate}
\end{enumerate}
No conditions are placed on $\xi$. 
\end{definition}

Given a C-structure $(\SSS,\FFF,\xi,\eta,\psi)$, a set $Z \subset X$, and  $\eps>0$, let
\[
\EE(Z,\eps) = \Big\{ \GGG \subset \SSS : Z \subset \bigcup_{s\in\GGG} U_s 
\text{ and } \psi(\GGG) \leq \eps \Big\}.
\]
By \cite[Proposition 1.1]{pes97}, the following defines an outer measure on $X$ for each $\alpha\in\RR$:
\begin{equation}\label{eqn:mCa}
m_C^\alpha(Z) := \lim_{\eps\to 0} \inf_{\GGG \in \EE(Z,\eps)} \sum_{s\in \GGG} \xi(s) \eta(s)^\alpha,
\end{equation}
where $m_C^\alpha(\emptyset):=0$. Then by \cite[Proposition 1.2]{pes97} there is a unique $\alpha_C \in [-\infty,\infty]$ such that $m_C^\alpha(X) = \infty$ for all $\alpha < \alpha_C$, and $m_C^\alpha(X) = 0$ for all $\alpha > \alpha_C$.  Writing $\dim_C(X) = \alpha_C$ for this \emph{Carath\'eodory dimension} of $X$, one is naturally led to the question of whether $m_C^{\alpha_C}(X)$ is positive and finite.  This requires further information about the C-structure, which we will explore more in the next section.

When $X$ is a metric space, one may ask whether $m_C := m_C^{\alpha_C}$ gives a Borel measure on $X$. Recall that an outer measure $m$ on a metric space $X$ is a \emph{metric outer measure} if $m(A\cup B) = m(A) + m(B)$ whenever $d(A,B) := \inf\{d(x,y) : x\in A, y\in B\} > 0$. Every metric outer measure has the property that all Borel sets are measurable; see \cite[Proposition 12.41]{Royden} or \cite[\S2.3.2(9)]{hF69}.

\begin{lemma}\label{lem:get-metric}
Suppose $\{\FFF,\xi,\eta,\psi\}$ is a C-structure with the property that for every $\rho>0$, there is $\eps>0$ such that $\psi(s) \leq \eps$ implies $\diam(U_s) < \rho$. Then $m_C = m_C^{\alpha_C}$ is a metric outer measure, and hence gives a Borel measure on $X$.
\end{lemma}
\begin{proof}
If $d(A,B)>0$, there is $\eps>0$ such that if $U_s$ intersects $A$, $U_t$ intersects $B$, and $\psi(s),\psi(t) \leq \eps$, then $U_s \cap U_t = \emptyset$, which means that every $\GGG \in \EE(A\cup B,\eps)$ splits into two disjoint subsets, one covering $A$ and the other covering $B$. This implies that $m_C(A\cup B) = m_C(A) + m_C(B)$.
\end{proof}

\section{Main results}\label{sec:results}

All the results in this section will be proved in \S\ref{sec:pf}.

\subsection{Reference measures on unstable leaves and transversals}\label{sec:measures}

From now on we assume that $\Lambda$ is a basic set for a flow $F=(f_t)_{t\in\RR}$ and $\ph \colon \Lambda\to \RR$ is a H\"older continuous potential function: that is, there is $\sigma>0$ such that
\begin{equation}\label{eqn:Holder}
|\ph|_\sigma := \sup%_{\substack{x,y\in \Lambda \\ x\neq y}} 
\Big\{ \frac{|\ph(x) - \ph(y)|}{d(x,y)^\sigma} : x,y\in\Lambda, x\neq y \Big\} < \infty.
\end{equation}
Given $x\in \Lambda$ and $t\in\RR$, we write\footnote{Observe that for $t<0$ this is the negative of the Birkhoff integral along the orbit segment from $f_t x$ to $x$: we have $\Phi(x,t) = -\Phi(f_t x, -t)$.}
\begin{equation}\label{eqn:Phi}
\Phi(x,t) := \int_0^t \ph(f_s x)\,ds.
\end{equation}
Our first task is to construct a \emph{system of $\Wcs$-transversal measures}: by this we mean a family of Borel measures $\{ \mmu_W : W \in \WWW^u\}$ with the following properties.
\begin{itemize}
\item For every $W\in \WWW^u$, the measure $\mmu_W$ is supported on $W\cap \Lambda$ in the sense that $\mmu_W(M \setminus (W\cap \Lambda)) = 0$.
\item If $W_1,W_2 \subset \WWW^u$ and $W_1 \subset W_2$, then $\mmu_{W_1}(Z) = \mmu_{W_2}(Z)$ for all Borel $Z\subset W_1$.
\end{itemize}
We will see in \S\ref{sec:scaling-properties} that the system of measures we construct also has important \emph{continuity} and \emph{conformality} properties, and in \S\ref{sec:product} that these properties actually characterize the system of measures up to a scalar multiple. 

Let $\delta_0>0$ be the size of the local manifolds as in Definition \ref{def:local-leaves}, and let $r_1>0$ be as in Lemma \ref{lem:push-trans}. We now fix once and for all a scale $0 < r < \min(r_1,\delta_0/3)$. 
Given $W\in\WWW^u$, define the $W$-Bowen ball centered at $x\in W$ with radius $r$ and order $t\geq 0$ by\footnote{This definition uses the \emph{intrinsic} metric on $W$ and thus differs slightly from \cite{CPZ,CPZ2}, where the leaf measures were defined using Bowen balls in the extrinsic metric from $M$.}
\begin{equation}\label{eqn:u-Bowen}
B_t(x,r,W) = \{y\in W : d_{f_\tau W} (f_\tau y, f_\tau x) < r \text{ for all } \tau \in [0,t] \}.
\end{equation}
Given $W\in\WWW^u$, define a C-structure by $X = W\cap \Lambda$, $\SSS = X\times [0,\infty)$, $U_{(x,t)} = B_t(x,r,W) \cap \Lambda$, $\xi(x,t) = e^{\Phi(x,t)}$, and $\eta(x,t) = \psi(x,t) = e^{-t}$. 
It turns out (for example, as a consequence of \eqref{eqn:mux-K} below) that the Carath\'eodory dimension $\dim_C(W\cap \Lambda)$ for this C-structure is equal to the topological pressure $P(\ph)$, and we will consider the corresponding outer measure on $W\cap \Lambda$ defined by \eqref{eqn:mCa} with $\alpha = P(\ph)$, which
can be described as follows. Given $Z\subset W \cap \Lambda$ and $T>0$, let
\begin{equation}\label{eqn:E}
\EE(Z,T) = \Big\{\EEE \subset (W\cap \Lambda) \times [T,\infty) : 
Z\subset \bigcup_{(x,t) \in \EEE} B_t(x,r,W) \Big\}.
\end{equation}
Members of $\EE(Z,T)$ can be thought of as collections of Bowen balls ``enveloping'' the set $Z$. Then we define $\mmu_W$ by
\begin{equation}\label{eqn:mu}
\mmu_W(Z) = \lim_{T\to\infty} \inf_{\EEE\in \EE(Z,T)} 
\sum_{(x,t) \in \EEE} e^{\Phi(x,t) - t P(\ph)}.
\end{equation}
We will prove in \S\ref{sec:balls} that for every $W\in \WWW^u_\theta$, $x\in W\cap \Lambda$, and $t\geq t_2=t_2(\theta)$, we have $\diam B_t(x,r,W) \leq 2r \lambda^{t-t_2}$; see \eqref{eqn:diam-t-eta}.
Thus this C-structure satisfies the hypothesis of Lemma \ref{lem:get-metric}, and we conclude that $\mmu_W$ is a Borel measure.\footnote{This is the step where it is crucial that we work on $W\cap \Lambda$ and not on $\Lambda$ itself; the outer measure defined by \eqref{eqn:mu} on $\Lambda$ is not a Borel measure. We will return to this point in \S\ref{sec:two-sided}.} 
Observe that different values of $r$ give different (but equivalent) measures $\mmu_W$, but since we fix $r$ throughout the paper, we will suppress this dependence from the notation.

Now we restrict to the case of local unstable leaves and assume that $W = \Wu(x,\delta)$ for some $x\in \Lambda$ and $\delta\in (0,\delta_0]$. In this case we write $\mmu_x = \mmu_W$, and we have the following.\footnote{Theorem \ref{thm:properties} holds for $W\in \WWW^u$ as well, but in that case cannot be deduced directly from \cite{CPZ2}, which considered only local unstable leaves. Rather, it follows from the case here together with the scaling properties in the next section. For $W\in \WWW^u$, the lower bound in \eqref{eqn:mux-K} depends not on the diameter of $W$ itself, but on the diameter of its projection to a local unstable leaf under holonomy.}
%In \S\ref{sec:prop-pf}, we will deduce the following from results in \cite{CPZ,CPZ2}.

\begin{theorem}\label{thm:properties}
Let $F=(f_t)_{t\in\RR},\Lambda,\ph,r$ be as above.
Given $\delta>0$, there is $K>0$ such that for every local unstable leaf $W = \Wu(x,\delta)$ with $x\in \Lambda$, the Borel measure $\mmu_x$ on $W$ defined by \eqref{eqn:mu} satisfies 
\begin{equation}\label{eqn:mux-K}
K^{-1} \leq \mmu_x(W) \leq K.
\end{equation}
Moreover, for every $x\in \Lambda$, the measures
\begin{equation}\label{eqn:nut}
\nu_t := \frac 1t \int_0^t (f_s)_* \mmu_x \,ds = \frac 1t \int_0^t \mmu_x \circ f_{-s} \,ds
\end{equation}
are weak* convergent as $t\to\infty$ to a limiting measure, whose normalization $\mu$ is independent of $x$. This measure $\mu$ is the unique equilibrium measure for $\ph$ on $\Lambda$. It is ergodic, gives positive weight to every (relatively) open set in $\Lambda$, and has the Gibbs property \eqref{eqn:Gibbs}.
\end{theorem}

We will prove Theorem \ref{thm:properties} in \S\ref{sec:prop-pf} using the corresponding results in \cite{CPZ2}, which cover partially hyperbolic diffeomorphisms and can be applied to the time-$\tau$ map $f_\tau$.
A more expository overview of the principles driving the proof is given in \cite[\S6]{CPZ}.

\subsection{Scaling properties}\label{sec:scaling-properties}

An important role in \cite{CPZ,CPZ2} is played by the behavior of the measures $\mmu_x$ under transformation by the dynamics and by holonomy maps. In our present setting we can strengthen the results proved there. Note that now we return to the general case $W\in\WWW^u$, instead of restricting our attention to local unstable leaves.

\begin{definition}\label{def:conformal}
We say that a system of $\Wcs$-transversal measures $\{ \bmu_W : W\in \WWW^u\}$ is \emph{$\ph$-conformal} if given any $t\in \RR$ and $W,f_t(W) \in \WWW^u$, the measures $(f_t)_* \bmu_W$ and $\bmu_{f_t(W)}$ are equivalent, and for each Borel $Z\subset W$ we have
\begin{equation}\label{eqn:pull-scale}
\bmu_{f_t(W)}(f_t Z) = \int_Z e^{tP(\ph) - \Phi(z,t)} \,d\bmu_W(z).
\end{equation}
In terms of the Radon--Nikodym derivatives this is equivalent to
\begin{align}\label{eqn:RN-conf}
\frac{d ((f_{-t})_* \bmu_{f_t(W)})}{d\bmu_W}(z) &= e^{tP(\ph) - \Phi(z,t)}, \\
\label{eqn:RN-conf-2}
\frac{d((f_t)_* \bmu_W)}{d\bmu_{f_t(W)}}(f_t z) &= e^{\Phi(z,t) - tP(\ph)}.
\end{align}
\end{definition}

\begin{definition}\label{def:cts}
Given $W_1,W_2 \in \WWW^u$, a \emph{weak-stable holonomy} between $W_1$ and $W_2$ is a homeomorphism $\pi \colon W_1 \cap \Lambda \to W_2 \cap \Lambda$ such that $\pi(z) \in \Wcs(z)$ for all $z\in W_1 \cap \Lambda$.
If $\pi(z) \in \Wcs(z,\delta)$ for all $z\in W_1\cap \Lambda$, then we say that $\pi$ is a \emph{weak-stable $\delta$-holonomy}.\footnote{Given $\delta>0$, it follows from the local product structure that if $x\in W_1$ and $y\in W_2$ are sufficiently close, then there are neighborhoods of $x,y$ in $W_1\cap \Lambda$ and $W_2\cap \Lambda$ that are related by a weak-stable $\delta$-holonomy.} Replacing $\Wcs$ by $\Ws$ gives the definition of \emph{strong-stable holonomy},\footnote{It is worth noting that strong-stable holonomies do not always exist, for example between $W$ and $f_t(W)$.} and reversing the roles of $u,s$ defines \emph{weak-unstable} and \emph{strong-unstable holonomies}.

We say that a system of $\Wcs$-transversal measures $\{ \bmu_W : W\in \WWW^u\}$ is
\emph{continuous} if for every $\eps>0$ there is $\delta>0$ such that
if $W_1,W_2$ are $u$-admissible manifolds with $d_{C^1}(W_1,W_2) < \delta$ that are related by a weak-stable $\delta$-holonomy $\pi$,
then for every Borel $Z\subset W_1 \cap \Lambda$, we have
\begin{equation}\label{eqn:cts}
\bmu_{W_2}(\pi Z) = e^{\pm\eps} \bmu_{W_1}(Z).
\end{equation}
\end{definition}

\begin{theorem}\label{thm:conf-cts}
Let $F,\Lambda,\ph,r$ be as in \S\ref{sec:measures},
and let $W\in\WWW^u$ and $t\in \RR$. Then the system of $\Wcs$-transversal measures $\{ \mmu_W : W\in \WWW^u\}$ defined in \eqref{eqn:E}--\eqref{eqn:mu} is $\ph$-conformal and continuous.
\end{theorem}

\begin{remark}
In the discrete-time setting of \cite{CPZ,CPZ2}, the natural analogue of \eqref{eqn:E}--\eqref{eqn:mu} defines a system of transversal measures that is $\ph$-conformal, but it is not known whether or not the system defined there is also continuous in the sense of \eqref{eqn:cts}; all that is proved is that this relationship holds for \emph{some} $\eps>0$, which suffices for the analogue of Theorem \ref{thm:properties}, but not for the remaining results that we prove here. See \S\ref{sec:alternate} for more details.
\end{remark}

\begin{remark}
We will eventually see in Corollary \ref{cor:unique} below that up to a scalar multiple, the measures defined in \eqref{eqn:E}--\eqref{eqn:mu} give the \emph{only} system of $\Wcs$-transversal measures that is $\ph$-conformal and continuous. 
Two special cases are worth highlighting.
\begin{itemize}
\item When $\ph\equiv 0$, we have $P(\ph) = \htop(F) =: h$, the topological entropy, and \eqref{eqn:pull-scale} gives $\mmu_{f_t(W)} = e^{th} (f_t)_* \mmu_W$, which is the scaling property satisfied by the Margulis measures on unstable leaves.
\item When $\phu$ is the \emph{geometric potential}
\begin{equation}\label{eqn:geom}
\phu(x) = -\lim_{t\to 0} \frac 1t \log \det (Df_t|_{\Eu_x})
\end{equation}
and $\Lambda$ is an attractor, it can be shown that $P(\phu)=0$ and that the unique equilibrium measure is the SRB measure on $\Lambda$ \cite[\S7.4]{FH19}.
Then $e^{\Phu(x,t)} = \det(Df_t|_{\Eu_x})^{-1}$ and \eqref{eqn:pull-scale} gives 
\begin{equation}\label{eqn:Leb-scale}
%\frac{d((f_t)_*\mmu_W)}{d\mmu_{f_t(W)}}(f_t z) = \det (Df_t|_{\Eu_z})^{-1},
\mmu_{f_t(W)}(f_t Z) = \int_Z \det (Df_t|_{\Eu_z}) \,d\mmu_W(z).
\end{equation}
which is the same scaling property satisfied by leaf volume.
\end{itemize}
Since the Margulis measures and leaf volumes also satisfy the continuity property, it will follow from the uniqueness result in Corollary \ref{cor:unique} below that in these two cases these systems of measures are in fact given by $\mmu_W$ (up to a global scaling constant).
\end{remark}

It turns out that the properties of $\ph$-conformality and continuity are enough to completely describe how a system of $\Wcs$-transversal measures transforms under weak-stable holonomies. We will restrict our attention to $\delta_0$-holonomies, since this is all we need later on.\footnote{In fact we will prove the following results separately for global holonomies along flow lines and along strong-stable manifolds, but to combine these results into a formula for weak-stable holonomies requires some care because $t$ in \eqref{eqn:omega+} may not be uniquely determined by $x$ and $y$ unless we work locally.}
%One could use the result here together with the $\ph$-conformality property to deduce results for global weak-stable holonomies, but some care is needed in writing down the details because $t$ in \eqref{eqn:omega+} may not be uniquely determined by $x$ and $y$.}

Given $x\in \Lambda$ and $y\in \Wcs(x,\delta_0) \cap \Lambda$, there is a unique $t=t(x,y)\in (-\delta_0,\delta_0)$ such that $f_t(x) \in \Ws(y,\delta_0)$. We define
\begin{equation}\label{eqn:omega+}
\omega^+(x,y) = \Phi(x,t) - tP(\ph) + \int_0^\infty (\ph(f_{\tau+t} x) - \ph(f_\tau y)) \,d\tau.
\end{equation}
The improper integral converges absolutely because $d(f_{\tau+t} x, f_\tau y) \to 0$ exponentially quickly and $\ph$ is H\"older continuous.

\begin{theorem}\label{thm:cs-hol}
Let $F,\Lambda,\ph$ be as in \S\ref{sec:measures},
%above, 
and let $\{ \bmu_W : W\in \WWW^u \}$ be any continuous and $\ph$-conformal system of $\Wcs$-transversal measures. Suppose that $W_1,W_2 \in \WWW^u$ are related by a weak-stable $\delta_0$-holonomy $\pi\colon W_1\cap \Lambda\to W_2 \cap \Lambda$. Then the measures $\pi_* \bmu_{W_1}$ and $\bmu_{W_2}$ are equivalent, and we have
\begin{equation}\label{eqn:cs-hol}
\frac{d(\pi_*\bmu_{W_1})}{d\bmu_{W_2}}(\pi(z)) = e^{\omega^+(z, \pi z)}.
\end{equation}
\end{theorem}

 It is worth noting the special cases
\begin{equation}\label{eqn:omega+Ws}
\omega^+(x,y) = \int_0^\infty (\ph(f_\tau x) - \ph(f_\tau y)) \,d\tau
\qquad\text{when } x\in \Ws(y),
\end{equation}
and for points on the same orbit,
\begin{equation}\label{eqn:omega-flow}
\omega^+(x,f_t x) = \Phi(x,t) - tP(\ph)
\end{equation}
In particular, when $W_2 = f_t(W_1)$ and $\pi=f_t$, \eqref{eqn:cs-hol} reduces to the $\ph$-conformality property \eqref{eqn:RN-conf-2}.

We will also use the cocycle relation
\begin{equation}\label{eqn:cocycle}
\omega^+(x,y) = \omega^+(x,z) + \omega^+(z,y)
\quad\text{for } y,z\in \Wcs(x,\delta_0/2)
\end{equation}
which 
can be proved by first observing that the definition \eqref{eqn:omega+} gives
\begin{equation}\label{eqn:omega+2}
\begin{aligned}
\omega^+(x,y) &= \lim_{T\to\infty} \big(\Phi(x,T+t(x,y)) - \Phi(y,T)\big) - t(x,y) P(\ph) \\
&= \lim_{T\to\infty}
\big(\Phi(x,T) - \Phi(y,T-t(x,y))\big) - t(x,y) P(\ph).
\end{aligned}
\end{equation}
Thus the two terms on the right-hand side of \eqref{eqn:cocycle} can be written as
\begin{align*}
\omega^+(x,z) &= \lim_{T\to\infty} \big(\Phi(x,T+t(x,z)) - \Phi(z,T) \big) - t(x,z) P(\ph), \\
\omega^+(z,y) &= \lim_{T\to\infty} \big(\Phi(z,T) - \Phi(y,T - t(z,y)) \big) - t(z,y) P(\ph);
\end{align*}
adding these and using the identity $t(x,y) = t(x,z) + t(z,y)$, we see that the right-hand side of \eqref{eqn:cocycle} is equal to
\begin{align*}
\lim_{T\to\infty} \big( \Phi(x,T+t(x,z)) - \Phi(y,T-t(z,y)) \big) - t(x,y) P(\ph).
\end{align*}
Writing $T' = T - t(z,y)$ so that $T + t(x,z) = T' + t(x,y)$, this becomes
\[
\lim_{T'\to\infty} \big( \Phi(x,T' + t(x,y)) - \Phi(y,T') \big) - t(x,y) P(\ph),
\]
which is equal to $\omega^+(x,y)$ by \eqref{eqn:omega+2}, thus proving \eqref{eqn:cocycle}.

\subsection{Reference measures on stable leaves and transversals}\label{sec:others}

Recall from Definition \ref{def:cs-trans} that $\WWW^s$ denotes the set of $C^1$ injectively immersed submanifolds $W\subset M$ that are transverse to the weak-unstable direction. Given $W\in \WWW^s$, we construct a Borel measure $\ms_W$ on $W\cap\Lambda$ by applying the construction in \S\ref{sec:measures} to the time-reversed flow. In terms of the original flow, this can be defined using \emph{backwards Bowen balls}: given $x\in \Lambda$ and $t\geq 0$, let
%\begin{equation}\label{eqn:back-bowen}
\[
B_t^-(x,r,W) = \{z\in M : d_{f_{-\tau} W}(f_{-\tau} z, f_{-\tau} x) < r \text{ for all } \tau \in [0,t]\}.
\]
%\end{equation}
Then given $Z\subset W \cap \Lambda$ and $T>0$, let
\begin{equation}\label{eqn:E-}
\EE^-(Z,T) = \Big\{\EEE \subset (W\cap \Lambda) \times [T,\infty) : 
Z\subset \bigcup_{(x,t) \in \EEE} B_t^-(x,r,W) \Big\}.
\end{equation}
Finally, define $\ms_W$ by
\begin{equation}\label{eqn:ms}
\ms_W(Z) = \lim_{T\to\infty} \inf_{\EEE\in \EE^-(Z,T)} 
\sum_{(x,t) \in \EEE} e^{\Phi(f_{-t}x,t) - t P(\ph)}.
\end{equation}
Observe that the Birkhoff integral in \eqref{eqn:ms} evaluates as
\[
\Phi(f_{-t} x, t) = \int_0^t \ph(f_{\tau -t} x) \,d\tau
= \int_{-t}^0 \ph(f_s x)\,ds
= -\Phi(x,-t);
\]
care must be taken with the signs because the $\Phi$ notation favors the initial point of the orbit segment, rather than the final one. Theorem \ref{thm:properties} shows that this defines a positive finite Borel measure on each local stable leaf,\footnote{As with $\mmu$, positivity and finiteness on more general compact $W\in \WWW^s$ will follow from this together with the results on scaling under holonomies.}
and that pushing \emph{backwards} and averaging gives a family of measures converging to a scalar multiple of the unique equilibrium measure.

Versions of the scaling properties in \S\ref{sec:scaling-properties} hold for the system of $\Wcu$-transversal measures $\{\ms_W : W\in \WWW^s\}$.

\begin{definition}
We say that a system of $\Wcu$-transversal measures $\{\bms_W : W\in \WWW^s\}$ is \emph{$\ph$-conformal} if given any $t\in \RR$ and $W,f_t(W) \in \WWW^s$, the measures $(f_t)_* \ms_W$ and $\ms_{f_t(W)}$ are equivalent, and for each Borel $Z\subset W$ we have
\begin{equation}\label{eqn:pull-scale-ms}
\bms_{f_t(W)}(f_t Z) = \int_Z e^{-tP(\ph) + \Phi(z,t)} \,d\bms_W(z).
\end{equation}
The system is \emph{continuous} if for every $\eps>0$ there is $\delta>0$ such that if $W_1,W_2$ are $s$-admissible manifolds with $d_{C^1}(W_1,W_2) < \delta$ that are related by a weak-unstable $\delta$-holonomy $\pi$, then for every Borel $Z\subset W_1 \cap \Lambda$, we have
\begin{equation}\label{eqn:cts-ms}
\bms_{W_2}(\pi Z) = e^{\pm \eps} \ms_{W_1}(Z).
\end{equation}
\end{definition}

Given $x\in \Lambda$ and $y\in \Wcu(x,\delta_0) \cap \Lambda$, there is a unique $t=t(x,y) \in (-\delta_0,\delta_0)$ such that $f_t(x) \in \Wu(y,\delta_0)$. We define
\begin{equation}\label{eqn:omega-}
\omega^-(x,y) := -\Phi(x,t) + tP(\ph) + \int_0^\infty( \ph(f_{t-\tau} x) - \ph(f_{-\tau} y) ) \,d\tau
\end{equation}
and observe that $\omega^-$ satisfies the cocycle equation \eqref{eqn:cocycle}. Note that when $y=f_t x$, we have 
\begin{equation}\label{eqn:omega+-}
\omega^-(x, f_t x) = -\Phi(x,t) + tP(\ph) = -\omega^+(x,f_t x).
\end{equation}
The following result is just Theorem \ref{thm:cs-hol} with time reversed.

\begin{theorem}\label{thm:cu-hol}
Let $F,\Lambda,\ph$ be as in \S\ref{sec:measures},
and let $\{\bms_W : W\in \WWW^s\}$ be any continuous $\ph$-conformal system of $\Wcu$-transversal measures. Suppose that $W_1,W_2 \in \WWW^s$ are related by a weak-unstable $\delta_0$-holonomy $\pi \colon W_1 \cap \Lambda \to W_2\cap \Lambda$. Then the measures $\pi_* \bms_{W_1}$ and $\bms_{W_2}$ are equivalent, and we have
\begin{equation}\label{eqn:cu-hol}
\frac{d(\pi_*\bms_{W_1})}{d\bms_{W_2}}(\pi(z)) = e^{\omega^-(z, \pi z)}.
\end{equation}
\end{theorem}

By Theorem \ref{thm:conf-cts} with time reversed, the system of $\Wcu$-transversal measures given by \eqref{eqn:E-} and \eqref{eqn:ms} is continuous and $\ph$-conformal, so it satisfies the conclusion of Theorem \ref{thm:cu-hol}.

\subsection{A product construction}\label{sec:product}

In this section we once again restrict our attention to local stable and unstable leaves. Let $\bW^s \subset \WWW^s$ denote the collection of local stable leaves, and $\bW^u \subset \WWW^u$ the collection of local unstable leaves.\footnote{Recall from Definition \ref{def:local-leaves} that $W\in \bW^s$ need not be of the form $\Ws(x,\delta)$: in general, any subset of a local stable leaf that is a topological ball of maximal dimension is itself a local stable leaf, and similarly for unstable leaves.} With $\delta_0>0$ as in Definition \ref{def:local-leaves}, we denote the collection of local weak-stable leaves by
\begin{equation}\label{eqn:Wcs}
\bW^{cs} := \Big\{ W_\delta := \bigcup_{t\in (-\delta,\delta)} f_t(W) : W\in \bW^s, \delta \in (0,\delta_0] \Big \},
\end{equation}
and similarly for the collection $\bW^{cu}$ of local weak-unstable leaves.

By a \emph{system of measures on local unstable leaves} we mean a family of Borel measures $\{ \bmu_W : W \in \bW^u\}$ satisfying the support and consistency properties from the beginning of \S\ref{sec:measures}, and similarly with `unstable' replaced by stable, weak-unstable, and weak-stable. The definitions of continuity and $\ph$-conformality for these systems remain the same as for the corresponding systems on transversals, and Theorems \ref{thm:cs-hol} and \ref{thm:cu-hol} go through here with identical proofs.

Now suppose $\{\bms_W : W\in \bW^s\}$ and $\{\bmu_W : W\in \bW^u\}$ are continuous and $\ph$-conformal systems of measures on local stable and unstable leaves. Define corresponding systems of measures on local weak-stable and weak-unstable leaves by
\begin{equation}\label{eqn:weak-measures}
\bmcs_{W_\delta} = \int_{-\delta}^\delta \bms_{f_t W} \,dt
\qquad\text{and}\qquad
\bmcu_{W_\delta} = \int_{-\delta}^\delta \bmu_{f_t W} \,dt.
\end{equation}
Let $L,\delta_1>0$ be as in Definition \ref{def:lps}, so that given any $x,y \in \Lambda$ with $d(x,y) \in \delta_1$, the bracket $[x,y]$ is well-defined as the unique point in $\Wcs(x,L\delta_1) \cap \Wu(y,L\delta_1)$. Thus given $q\in \Lambda$ and $\delta \in (0,\delta_1/2)$ and writing
\[
\Ru_q := \Wu(q,\delta) \cap \Lambda,
\qquad
\Rcs_q := \Wcs(q,\delta) \cap \Lambda,
\]
we can define
\[
R_q := \{ [x,y] : x\in \Ru_q, y\in \Rcs_q\}.
\]
Given $x\in \Ru_q$ and $y\in \Rcs_q$, we write $\Psi_q(x,y) = [x,y]$, so that $\Psi_q$ is a homeomorphism between $\Ru_q\times \Rcs_q$ and $R_q$.

From now on we fix $\delta$ sufficiently small that $\diam R_q < \delta_1$, and thus the bracket is well-defined for every pair of points in $R_q$. For this fixed value of $\delta$, we write
\begin{equation}\label{eqn:bmu-cs}
\bmu_q := \bmu_{\Wu(q,\delta)}
\qquad\text{and}\qquad
\bmcs_q := \bmcs_{\Wcs(q,\delta)}.
\end{equation}
Given $z = \Psi_q(x,y) \in R_q$, we write
\[
\Ru_q(z) := \Wu(z,L\delta) \cap R_q,
\qquad
\Rcs_q(z) := \Wcs(z,L\delta) \cap R_q,
\]
and we observe that $z\in \Rcs_q(x) \cap \Ru_q(y)$,
so that $x\in \Rcs_q(z) \cap \Ru_q(q)$ and $y\in \Rcs_q(q) \cap \Ru_q(z)$.  In other words,
\begin{equation}\label{eqn:xyz}
z = [x,y],\qquad
x = [z,q],\qquad
y = [q,z].
\end{equation}
Now we can formulate the main result on the product construction of the equilibrium measure.

\begin{theorem}\label{thm:product}
Let $F,\Lambda,\ph,\delta$ be as in \S\ref{sec:measures},
and let $\{\bmu_x,\bmcs_x : x\in \Lambda\}$ be given by \eqref{eqn:weak-measures} and \eqref{eqn:bmu-cs} from any continuous $\ph$-conformal systems of measures on local stable and unstable leaves.
Then for every $q\in \Lambda$, the following equations all define the same measure on $R_q$: given a Borel set $Z \subset R_q = \Psi_q(\Ru_q \times \Rcs_q)$, put
\begin{align}
\label{eqn:muq}
\mu_q(Z) &= \int_Z e^{\omega^+(z,[z,q]) + \omega^-(z,[q,z])} \,d((\Psi_q)_*(\bmu_q\times \bmcs_q))(z), \\
\label{eqn:muq2}
\mu_q(Z) &= \int_{\Rcs_q} \int_{\Ru_q} e^{\omega^+([x,y],x) + \omega^-([x,y],y)}  
\mathbf{1}_Z ([x,y]) \,d\bmu_q(x) \,d\bmcs_q(y),\\
\label{eqn:muq3}
\mu_q(Z) &= \int_{\Rcs_q} \int_{\Ru_q(y) \cap Z}
e^{\omega^-(z,y)} \, d\bmu_y(z) \,d\bmcs_q(y), \\
\label{eqn:muq4}
\mu_q(Z) &= \int_{\Ru_q} \int_{\Rcs_q(x) \cap Z}
e^{\omega^+(z,x)} \, d\bmcs_x(z) \,d\bmu_q(x).
\end{align}
Moreover, there is a unique Borel measure $\mu$ on $\Lambda$ such that $\mu(Z) = \mu_q(Z)$ for all $q\in \Lambda$ and $Z\subset R_q$. The measure $\mu$ is flow-invariant, nonzero, and finite. It has the Gibbs property and thus is a scalar multiple of the unique equilibrium measure.
\end{theorem}

If one redefines the bracket using $\Ws$ and $\Wcu$ instead of $\Wu$ and $\Wcs$, then one obtains an analogue of Theorem \ref{thm:product} with $u$ replaced by $cu$, and $cs$ replaced by $s$, everywhere they appear.

The following is an immediate consequence of \eqref{eqn:muq3} and the definition of conditional measures (see \cite[\S14.2]{yC16} or \cite[\S\S5.1--5.2]{VO16}).
%\cite[\S2.3]{CPZ2})

\begin{corollary}\label{cor:conditionals}
Given $q\in \Lambda$, consider the partition of $R_q$ into local unstable leaves:
let $\hat\mu_q$ be the transversal measure on $\Rcs_q$ defined by $\hat\mu_q(Y) = \mu(\bigcup_{y\in Y} \Ru_q(y))$, and let $\{\mu_y^u : y\in \Rcs_q\}$ denote the conditional measures of $\mu$ characterized (for $\hat\mu_q$-a.e.\ $y$) by
\begin{equation}\label{eqn:conditionals}
\mu(Z) = \int_{\Rcs_q} \mu_y^u(Z\cap \Ru_q(y)) \,d\hat\mu_q(y).
\end{equation}
Then $\hat\mu_q$ and $\bmcs_q$ are equivalent, and for $\hat\mu_q$-a.e. $y$ (hence $\bmcs_q$-a.e. $y$) the measures $\mu_y^u$ and $\bmu_y$ are equivalent as well. Moreover, defining $h\colon \Rcs_q \to (0,\infty)$ by
\begin{equation}\label{eqn:hy}
h(y) := \int_{\Ru_q} e^{\omega^+([x,y],x)} \,d\bmu_q(x)
= \bmu_y(\Ru_q(y)),
\end{equation}
where the two expressions are equal by Theorem \ref{thm:cs-hol}, we have
\begin{equation}\label{eqn:mu-m}
\frac{d\hat\mu_q}{d\bmcs_q}(y) = h(y)
\quad\text{and}\quad
\frac{d\mu_y^u}{d\bmu_y}(z) = \frac{e^{\omega^-(z,y)}}{h(y)}.
\end{equation}
Corresponding results hold for the transversal and conditional measures associated to the partitions of $R_q$ by local stable, weak-stable, and weak-unstable leaves.
\end{corollary}

We prove the following corollary in \S\ref{sec:unique}.

\begin{corollary}\label{cor:unique}
If $\{ \bmu_W : W\in \bW^u\}$ is any continuous $\ph$-conformal system of measures on local unstable leaves, then there is a real number $c>0$ such that $\bmu_W = c\mmu_W$ for every $W\in \bW^u$, where $\mmu_W$ is the leaf measure defined in \eqref{eqn:E}--\eqref{eqn:mu}. Similar results hold for local stable, weak-stable, and weak-unstable leaves.
\end{corollary}

We return to the case when $\Lambda$ is an attractor and $\ph=\phu$ is the geometric potential as in \eqref{eqn:geom}, so that $P(\ph)=0$, $e^{\Phi^u(y,t)} = \det(Df_t|_{\Eu_y})^{-1}$, and given $y\in \Lambda$, $z\in \Wu(y,\delta)$, we have
\begin{equation}\label{eqn:omega-geom}
e^{\omega^-(z,y)} = \lim_{t\to-\infty} \frac{e^{\Phi^u(z,t)}}{e^{\Phi^u(y,t)}}
= \lim_{t\to -\infty} \frac{\det(Df_t|_{\Eu_y})}{\det(Df_t|_{\Eu_z})}.
\end{equation}
Let $\nnu_y$ be the Borel measure on $\Wu(y,\delta)$ given by
\begin{equation}\label{eqn:nuy}
\frac{d\nnu_y}{d\Leb_{\Wu(y,\delta)}}(z) = \lim_{t\to -\infty} \frac{\det(Df_t|_{\Eu_y})}{\det(Df_t|_{\Eu_z})},
\end{equation}
where $\Leb_{\Wu(y,\delta)}$ denotes leaf volume. Given $q\in \Lambda$, define a Borel measure $\ns_q$ on $\Ws(q,\delta) \cap \Lambda$ by
\begin{equation}\label{eqn:ns}
\ns_q(Z) = 
\lim_{T\to\infty} \inf_{\EEE \in \EE^-(Z,T)}
\sum_{(y,t) \in \EEE} \det(Df_{-t}|_{\Eu_y}),
\end{equation}
where $\EE^-(Z,T)$ is as in \eqref{eqn:E-}. Using this, define a measure $\ncs_q$ on $\Rcs_q = \Wcs(q,\delta) \cap \Lambda$ by 
\begin{equation}\label{eqn:ncs}
\ncs_q(Z) = \int_{-\delta}^\delta \ns_q(f_t Z) \,dt.
\end{equation}
Then \eqref{eqn:muq3} gives the following.

\begin{corollary}\label{cor:srb-product}
Let $\Lambda$ be a hyperbolic attractor and $\mu$ its associated SRB measure. Given $q\in \Lambda$, there exists $c=c(q)>0$ such that
\begin{equation}\label{eqn:srb-product}
\mu(Z) = c \int_{\Rcs_q} \nnu_y(Z \cap \Ru_q(y)) \,d\ncs_q(y)
\end{equation}
for every Borel $Z\subset R_q$, where $\nnu_y,\ncs_q$ are given by \eqref{eqn:nuy}--\eqref{eqn:ncs}.
\end{corollary}

One can write similar analogues of \eqref{eqn:muq}, \eqref{eqn:muq2}, and \eqref{eqn:muq4}. The fact that the measures $\nnu_y$ are the conditional measures of the SRB measure is well-known \cite[Lemma 2.3]{jS68}. To the best of the author's knowledge, the explicit description of the transverse measure $\ncs_q$ in \eqref{eqn:ns}--\eqref{eqn:ncs} is new, as is the corresponding description of the stable conditional measures in terms of $\ns_q$ that can be obtained using Corollary \ref{cor:conditionals} (though we do not write out these details).

\subsection{Refining in all directions}\label{sec:two-sided}

We conclude our main results by giving one more description of the unique equilibrium measure via a Carath\'eodory dimension structure, this time on $\Lambda$ itself rather than on a local stable or unstable leaf. In what follows we do \emph{not} assume that the metric is adapted.

To define a C-structure with $X=\Lambda$ and obtain a Borel measure via Lemma \ref{lem:get-metric}, we cannot use the one-sided Bowen balls we have used so far, because $\bigcap_{t>0} B_t(x,r)$ is a local weak-stable manifold, rather than a single point; similarly, $\bigcap_{t>0} B^-_t(x,r)$ is a local weak-unstable manifold. We can make some progress by using \emph{two-sided Bowen balls}
\begin{align*}
B^\pm_{s,t}(x,r) &:= B_t(x,r) \cap B^-_s(x,r) \\
&=\{ y\in M : d(f_\tau x, f_\tau y) < r \text{ for all } \tau \in [-s,t] \},
\end{align*}
for which taking an intersection over all $s,t>0$ gives a small piece of the orbit of $x$. Thus given a sufficiently small embedded disc $D$ that is transverse to the flow direction, we have $\bigcap_{s,t>0} B^\pm_{s,t}(x,r) \cap D= \{x\}$ for each $x\in D \cap\Lambda$, and so we could use two-sided Bowen balls to define a C-structure on $D \cap \Lambda$ that gives a Borel measure; then we could flow this measure forward to get an invariant measure on $\Lambda$. 

Rather than describing the details of this, we go one step further: provided we choose $r\in (0,\delta_1]$, we can use the local product structure in Definition \ref{def:lps} whenever $x,y\in \Lambda$ satisfy $d(x,y) < r$ to obtain $z = [x,y] \in \Wcs(x,Lr) \cap \Wu(y,Lr)$ and $z' = f_\beta(z) \in \Ws(x,Lr)$ for some $\beta = \beta(x,y) \in (-Lr,Lr)$. Then we make the following definition for $x\in \Lambda$ and $s,t>0$:
\begin{equation}\label{eqn:B*}
B^*_{s,t}(x,r) := \{ y \in B^\pm_{s,t}(x,r) \cap \Lambda : |\beta(x,y)| < r/(s+t) \}.
\end{equation}
Now we consider the C-structure defined by $X=\Lambda$, $\SSS=X\times[1,\infty)^2$, and $U_{(x,s,t)} = B^*_{s,t}(x,r)$, together with
\[
\xi(x,s,t) = \frac{e^{\Phi(f_{-s} x, s+t)}}{s+t},
\quad
\eta(x,s,t) = e^{-(s+t)},
\quad
\psi(x,s,t) = e^{-\min(s,t)}.
\]
The conditions of Lemma \ref{lem:get-metric} are satisfied, so this C-structure produces a Borel measure, which we describe more carefully below.

\begin{remark}\label{rmk:motivate-C}
As motivation for the formula for $\xi$,
observe that $B_{s,t}^*(x,r)$ is roughly the `product' of $B_t(y,r)$ in the unstable direction, $B_s^-(x,r)$ in the stable direction, and an interval in the flow direction: we want to produce a measure $m$ with a product structure whose conditionals give weight $e^{\Phi(x,t)}$ along the unstable, $e^{\Phi(f_{-s}x,s)}$ along the stable, and $1/(s+t)$ along the flow; the product of these weights gives $\xi$. 
Note that the factor of $1/(s+t)$ has no thermodynamic meaning and in particular does not depend on the potential; it is roughly proportional to the amount of time a trajectory takes to cross $B_{s,t}^*(x,r)$, and is present to guarantee that the measure $m$ defined by \eqref{eqn:m} below is flow-invariant.

The formula for $\eta$ is as in our previous C-structures: $e^{-\ell}$, where $\ell$ is the length of the orbit segment that defines the Bowen ball we use.
The formula for $\psi$ guarantees that both $s$ and $t$ are large when $\psi$ is small, which mimics our previous constructions. 

We remark that it would be tempting to simplify this C-structure by requiring that $s=t$, so that we use symmetric two-sided Bowen balls $B_t(x,r) \cap B^-_t(x,r)$;
however, our proof of flow-invariance requires us to allow the case $s\neq t$.
\end{remark}

This C-structure defines a Borel measure $m$ on $\Lambda$ as follows:
\begin{align}
\label{eqn:E*}
\EE^*(Z,T) &:= \Big\{ \EEE \subset \Lambda \times [T,\infty)^2 : Z \subset \bigcup_{(x,s,t) \in \EEE} B^*_{s,t}(x,r) \Big \}, \\
\label{eqn:m}
m(Z) &:= \lim_{T\to\infty} \inf_{\EEE \in \EE^*(Z,T)} \sum_{(x,s,t) \in \EEE} \frac 1{s+t} e^{\Phi(f_{-s} x, s+t) - (s+t)P(\ph)}.
\end{align}

\begin{theorem}\label{thm:direct}
Let $F,\Lambda,\ph,r$ be as in \S\ref{sec:measures}.
Then the Borel measure defined by \eqref{eqn:E*}--\eqref{eqn:m} is nonzero, finite, flow-invariant, and a scalar multiple of the unique equilibrium measure.
\end{theorem}

Observe that Theorem \ref{thm:srb-main} is a direct consequence of Theorem \ref{thm:direct} in the case when $\Lambda$ is an attractor and $\ph=\phu$ is the geometric potential.

\begin{remark}
Returning to the approach in \S\ref{sec:product}, one could adapt the construction here to use ``one-sided Bowen balls'' on weak-unstable leaves that also include a restriction on $|\beta(x,y)|$ as in \eqref{eqn:B*}, and thus obtain the measures $\mcu_x$ (and similarly $\mcs_x$) directly from a C-structure rather than first obtaining $\mmu_x$ (or $\ms_x$) and pushing under the flow.
\end{remark}

\begin{remark}
The precise form of the factor $(s+t)^{-1}$ in the definition of $B^*$ and of $m$ is not essential: one could use any function $\zeta(s+t)$ where $\lim_{T\to\infty} \zeta(T) = 0$ and $\llim_{T\to\infty} \zeta(T+\eta)/\zeta(T) > 0$
for every $\eta>0$
-- the latter bound is needed at the end of the proof of Lemma \ref{lem:*-gibbs}. Replacing $r/(s+t)$ in \eqref{eqn:B*} with $\zeta(s+t)$, and defining $\xi(x,s,t) = e^{\Phi(f_{-s} x, s+t)} \zeta(s+t)$, the proofs still go through.

One could also avoid using such a function by using an extra parameter in the definition of the C-structure and putting
\begin{align*}
\SSS &= \Lambda \times [1,\infty)^2 \times (0,1),\\
U_{(x,s,t,\delta)} &= \{ y\in B^{\pm}_{s,t}(x,r) \cap \Lambda : |\beta(x,y)| < \delta \},
\end{align*}
together with the functions
\begin{align*}
\xi(x,s,t,\delta) &= \delta e^{\Phi(f_{-s} x, s+t)}, \\
\eta(x,s,t,\delta) &= e^{-(s+t)}, \\
\psi(x,s,t,\delta) &= \max( e^{-s}, e^{-t}, \delta).
\end{align*}
It is a matter of taste which route one takes; here we have opted for the definition with fewer parameters and the explicit choice $\zeta(T) = 1/T$.
\end{remark}

\section{Proofs}\label{sec:pf}

\subsection{Bowen balls on leaves}\label{sec:balls}

Recall from \eqref{eqn:u-Bowen} that we define Bowen balls on transversals according to the intrinsic metric:
\[
B_t(x,r,W) = \{y\in W : d_{f_\tau W}(f_\tau y, f_\tau x) < r \text{ for all } \tau\in [0,t]\}.
\]
Now \eqref{eqn:u-admissible} immediately gives the following.

\begin{lemma}\label{lem:W-Bowen}
If $W$ is a $u$-admissible manifold and $r\in (0,r_0]$, then for all $x\in W \cap \Lambda$ and $t\geq 0$ we have
\begin{equation}\label{eqn:W-Bowen}
B_t(x,r,W) = f_{-t} B(f_t x, r, f_t W) \subset B(x, r\lambda^t, W).
\end{equation}
\end{lemma}

We also get a version of this lemma for $W\in \WWW^u$.
Let $r_1$ and $t_1=t_1(\theta)$ be as in Lemma \ref{lem:push-trans}, and let
\begin{equation}\label{eqn:some-growth}
C_1 = C_1(\theta) := \sup \{ \|Df_t(x)\| : x\in M, t\in [-t_1,t_1]\}.
\end{equation}

\begin{lemma}\label{lem:trans-Bowen}
For every $\theta>0$, there is $t_2\geq 0$ such that given any $r\in (0,r_1]$, $W\in \WWW^u_\theta$, any $x\in W\cap \Lambda$, and any $t\geq t_2$, we have
\begin{equation}\label{eqn:trans-Bowen}
B_t(x,r,W) = f_{-t} B(f_t x, r, f_t W).
\end{equation}
and $B(f_t x, r, f_t W)$ is $u$-admissible. Moreover, given $\eta>0$, we have
\begin{equation}\label{eqn:W-Bowen-2}
\begin{aligned}
B_{t+\eta}(x,r,W) &= f_{-t} B_\eta(f_t x, r, f_t W) \\
&\subset f_{-t} B(f_t x, r\lambda^\eta, f_t W) = B_t(x,r\lambda^\eta, W),
\end{aligned}
\end{equation}
and in particular
\begin{equation}\label{eqn:diam-t-eta}
\diam B_{t+\eta}(x,r,W) \leq 2r\lambda^\eta.
\end{equation}
\end{lemma}
\begin{proof}
Clearly the left-hand side of \eqref{eqn:trans-Bowen} is contained in the right-hand side, so it suffices to prove the other inclusion.
By Lemma \ref{lem:push-trans}, there is $t_1\geq 0$ such that for all $r,W,x$ as in the lemma, and any $t\geq t_1$, $B(f_t x, r, f_t W)$ is $u$-admissible. Then it follows from Lemma \ref{lem:W-Bowen} that for all such $t$ and every $\tau \in [t_1, t]$, we have
\[
f_{-(t-\tau)} B(f_t x, r, f_t W) \subset B(f_{\tau} x, r \lambda^{t-\tau}, f_{\tau} W).
\]
In particular, we have
\[
f_{-(t-t_1)} B(f_t x, r, f_t W) \subset B(f_{t_1} x, r \lambda^{t-t_1}, f_{t_1} W),
\]
and using this together with \eqref{eqn:some-growth} we see that for every $\tau \in [0,t_1]$ we have
\[
f_{-(t-\tau)} B(f_t x, r, f_t W) \subset B(f_\tau x, C_1 r \lambda^{t-t_1}, f_\tau W).
\]
Thus by choosing $t_2$ sufficiently large that $C_1 \lambda^{t_2 - t_1} \leq 1$, we guarantee that
\[
f_\tau (f_{-t} B(f_t x, r, f_t W)) \subset B(f_\tau x, r, f_\tau W)
\]
for all $t\geq t_2$ and $\tau \in [0,t]$. This shows that the right-hand side of \eqref{eqn:trans-Bowen} is contained in the left-hand side.

For \eqref{eqn:W-Bowen-2}, we apply \eqref{eqn:trans-Bowen} as follows:
\[
B_{t+\eta}(x,r,W) = \bigcap_{\tau \in [0,\eta]} B_t(f_\tau x, r, W)
= \bigcap_{\tau \in [0,\eta]} f_{-t} B(f_\tau(f_t x), r, f_{\tau+t} W)
\]
and observe that the final intersection is $f_{-t} B_\eta(f_t x, r, f_t W)$, which proves the first equality in \eqref{eqn:W-Bowen-2}. The inclusion and the final equality follow from Lemma \ref{lem:W-Bowen}. Then \eqref{eqn:diam-t-eta} follows immediately upon observing that $B_t(x,r\lambda^\eta, W) \subset B(x,r\lambda^\eta)$.
\end{proof}

\subsection{Finiteness and pushforwards: proof of Theorem \ref{thm:properties}}\label{sec:prop-pf}

We will deduce Theorem \ref{thm:properties} from the corresponding results in \cite{CPZ2}, which deal with partially hyperbolic diffeomorphisms. Thus the proof here should be read less with the intent of understanding the mechanisms driving Theorem \ref{thm:properties}, and more with the intent of understanding the relationship between the continuous- and discrete-time settings. For a description of the general mechanisms, we refer to \cite[\S6]{CPZ}.

\subsubsection{Relationship to discrete-time results and uniform bounds}\label{sec:discrete}

Before proving the bounds in \eqref{eqn:mux-K}, we describe the relationship to \cite{CPZ2}, observing that every time-$\tau$ map $f_\tau$ ($\tau>0$) of a flow on a hyperbolic set $\Lambda$ is partially hyperbolic.
The results in \cite{CPZ2} apply to a diffeomorphism $f$ with a compact invariant set $\Lambda$ on which $f$ is partially hyperbolic and topologically transitive, with a local product structure between unstable and center-stable foliations, and for which the center-stable foliation satisfies a certain ``Lyapunov stability'' condition 
%in Lemma \ref{lem:Lyap} (with $f_t$ replaced by $f^n$).
that is automatic for the time-$\tau$ map of a flow.
In order to apply the results in \cite{CPZ2} to $f_\tau$ we only need to check topological transitivity of $f_\tau$.
For this we use the following lemma, which may be of independent interest (though we do not claim that it is new).

\begin{lemma}\label{lem:map-transitive}
Let $X$ be a compact metric space and $(f_t)_{t\in\RR}$ a continuous flow on $X$ that is topologically transitive: for every open $U,V \subset X$ there is $t>0$ such that $f_t(U) \cap V \neq\emptyset$. Then there is a residual set of $\tau\in (0,\infty)$ such that the time-$\tau$ map $f_\tau$ is topologically transitive.
\end{lemma}
\begin{proof}
Note that given any open $U,V\subset X$, the set
\[
I(U,V) := \{ \tau > 0 : f_{k\tau}(U) \cap V \neq \emptyset \text{ for some } k \in \NN \}
\]
is open. We claim that it is dense in $(0,\infty)$. Indeed, given any $(a,b) \subset (0,\infty)$, if $k\geq b/(b-a)$ then $kb-ka \geq b > a$, so $(k+1)a<kb$ and therefore $(ka,kb) \cap ((k+1)a,(k+1)b) \neq \emptyset$; this implies that
\begin{equation}\label{eqn:intervals-cover}
\bigcup_{k\in \NN} (ka,kb) \supset \Big[\frac{ab}{b-a},\infty\Big).
\end{equation}
Since the flow is transitive there exists $t\geq ab/(b-a)$ such that $f_t(U) \cap V \neq \emptyset$ (note that $t$ can be chosen arbitrarily large by iterating the definition of transitivity). By \eqref{eqn:intervals-cover} there is $k\in\NN$ such that $t/k\in (a,b)$. Since $t/k \in I(U,V)$, this shows that $I(U,V)$ is open and dense in $(0,\infty)$.

Now let $\{U_n\}_{n\in\NN}$ be a basis for the topology on $X$. Any $\tau$ in the residual set $\bigcap_{k,n} I(U_k,U_n)$ has the property that $f_\tau$ is topologically transitive. Note that this set is dense in $(0,\infty)$ by the Baire category theorem.
\end{proof}

Returning to the setting of Theorem \ref{thm:properties}, use Lemma \ref{lem:map-transitive} to choose $\tau>0$ such that the time-$\tau$ map $f_\tau$ is topologically transitive, and thus writing $\psi(y) := \Phi(y,\tau) = \int_0^\tau \ph(f_t y) \,dt$, the conditions of \cite{CPZ2} are satisfied by $(\Lambda,f_\tau,\psi)$. Observe that
\begin{equation}\label{eqn:sum=int}
S_n^{f_\tau} \psi(y) := \sum_{k=0}^{n-1} \psi(f_\tau^k y) = \int_0^{\tau n} \ph(f_t y) \,dt = \Phi(y,\tau n),
\end{equation}
and the topological pressures of the map $f_\tau$ and the flow $F$ are related by \cite[Corollary 4.12(iii)]{pW75}:
\begin{equation}\label{eqn:2P}
P(f_\tau,\psi) = \tau P(\ph). 
\end{equation}
(We continue to write $P(\ph) = P(F,\ph)$ for the pressure on the flow.)
The construction in \cite[\S3.2]{CPZ2} defines a Borel measure $m_x^{\tau,r}$ on $W=\Wu(x,\delta)$ as follows: we work with discrete-time Bowen balls\footnote{Observe that these Bowen balls are defined with respect to the metric on $M$, rather than the leaf metric on $W$ as in \eqref{eqn:u-Bowen}.}
\[
B_n^{f_\tau}(y,r) := \{ z : d(f_\tau^ky, f_\tau^k z) < r \text{ for all } 0\leq k < n \},
\]
for which the corresponding collection of enveloping sets is
\[
\EE_\tau^r(Z,N) := \Big\{ \EEE \subset (W\cap \Lambda) \times \{N,N+1,\dots\}
: Z \subset \bigcup_{(y,n) \in \EEE} B_n^{f_\tau}(y,r) \Big\},
\]
and then produce a measure by
\begin{align}\label{eqn:mxCN}
m_x^{\tau,r}(Z,N) &:= \inf_{\EEE \in \EE_\tau^r(Z,N)} \sum_{(y,n) \in \EEE} e^{S_n^{f_\tau} \psi(y) - nP(f_\tau,\psi)}, \\
\label{eqn:mxC}
m_x^{\tau,r}(Z) &:= \lim_{N\to\infty} m_x^{\tau,r}(Z,N).
\end{align}
By \eqref{eqn:sum=int} and \eqref{eqn:2P}, we have
\begin{equation}\label{eqn:summands}
e^{\Phi(N\tau) - n\tau P(\ph)} = e^{S_n^{f_\tau} \psi(y) - nP(f_\tau, \psi)},
\end{equation}
which will let us relate the definitions of $\mmu_x$ in \eqref{eqn:mu} and $m_x^{\tau,r}$ in \eqref{eqn:mxC} by comparing $\EE(Z,T)$ and $\EE_\tau^r(Z,N)$. To avoid confusion we will write $B_t^F(y,r,W)$ for the continuous-time Bowen ball from \eqref{eqn:u-Bowen} in the remainder of this argument.

First observe that for each $n\in \NN$, we have $B_{\tau n}^F(y,r,W) \subset B_n^{f_\tau}(y,r)$. Thus given $\EEE \in \EE(Z,\tau N)$, we have
\[
Z \subset \bigcup_{(y,t) \in \EEE} B_t^F(y,r,W)
\subset \bigcup_{(y,t) \in \EEE} B_{\lfloor t/\tau \rfloor \tau}^F(y,r,W)
\subset \bigcup_{(y,t) \in \EEE} B_{\lfloor t/\tau \rfloor}^{f_\tau}(y,r),
\]
and thus by \eqref{eqn:summands},
\[
m_x^{\tau,r}(Z,N) 
\leq \sum_{(y,t) \in \EEE}
e^{S_{\lfloor t/\tau \rfloor}^{f_\tau}\psi(y) - \lfloor t/\tau \rfloor P(f_\tau,\psi)}
= \sum_{(y,t) \in \EEE} e^{\Phi(y,\lfloor t/\tau \rfloor \tau) - \lfloor t/\tau \rfloor \tau P(\ph)}.
\]
Observe that given $t \in [n\tau, (n+1)\tau]$ we have
\[
\Phi(y,n\tau) - \Phi(y,t) = -\Phi(f_{n\tau} y, t-n\tau) \leq \tau \|\ph\|,
%\int_0^{n\tau}\ph(f_s y)\,ds - \int_0^t \ph(f_s y)\,ds = -\int_{n\tau}^t \ph(f_s y) \leq \tau \|\ph\|,
\]
and similarly
\[
-n\tau P(\ph) + tP(\ph) = (t-n\tau) P(\ph) \leq \tau |P(\ph)|,
\]
so we obtain
\[
m_x^{\tau,r}(Z,N) \leq \sum_{(y,t) \in \EEE} e^{\tau(\|\ph\| + |P(\ph)|)} e^{\Phi(y,t) - tP(\ph)}.
\]
Taking an infimum over all $\EEE \in \EE(Z,\tau N)$ and sending $N\to\infty$ gives
\begin{equation}\label{eqn:mxCleq}
m_x^{\tau,r}(Z) \leq C_2^\tau \mmu_x(Z),
\quad\text{where } C_2 = e^{\|\ph\| + |P(\ph)|}.
\end{equation}
For a bound in the other direction, observe that by uniform continuity of the flow there is $\rho>0$ such that 
\begin{equation}\label{eqn:rho}
d(y,z) < \rho \text{ implies }d_{f_t W}(f_t y, f_t z) < r
\text{ for all }t\in [0,\tau].
\end{equation}
Then $B_n^{f_\tau}(y,\rho) \subset B_{\tau n}^F(y,r,W)$, so given any $\EEE \in \EE_\tau^\rho(Z,N)$, we have
\[
Z \subset \bigcup_{(y,n) \in \EEE} B_n^{f_\tau}(y,\rho) \subset \bigcup_{(y,n) \in \EEE} B_{\tau n}^F(y,r,W),
\]
and thus defining $\mmu_x(Z,T)$ as the infimum in \eqref{eqn:mu} (without taking the limit), we have
\[
\mmu_x(Z,N\tau) \leq \sum_{(y,n) \in \EEE} e^{\Phi(y,n\tau) - n\tau P(\ph)}
= \sum_{(y,n) \in \EEE} e^{S_n^{f_\tau} \psi(y) - nP(f_\tau,\psi)}. 
\]
Taking an infimum over all $\EEE \in \EE_\tau^\rho(Z,N)$ and sending $N\to\infty$ gives $\mmu_x(Z) \leq m_x^{\tau,\rho}(Z)$, and combining this with \eqref{eqn:mxCleq} we obtain
\begin{equation}\label{eqn:mCmu}
C_2^{-\tau} m_x^{\tau,r}(Z) \leq \mmu_x(Z) \leq m_x^{\tau,\rho}(Z) \text{ for all } Z \subset \Wu(x,\delta).
\end{equation}
Now the bounds in \eqref{eqn:mux-K} follow immediately from the corresponding bounds for $m_x^{\tau,r}$ and $m_x^{\tau,\rho}$ in \cite[Theorem 4.2]{CPZ2}, together with \eqref{eqn:mCmu}.

We record the following consequence of \eqref{eqn:mCmu} for use in the next section: for every $x\in \Lambda$ and $\tau \in (0,1]$ such that $f_\tau$ is transitive, writing $W = \Wu(x,\delta)$, we have
\begin{equation}\label{eqn:leqK}
m_x^{\tau,r}(W) \leq C_2^\tau \mmu_x(W) \leq C_2 \mmu_x(W) =: C_3(r) < \infty.
\end{equation}
We will use this below with $r$ replaced by $\rho$.

%if $\rho_1>0$ is such that $d(y,z) < \rho_1$ implies $d(f_t y, f_t z) < r$ for all $t\in [0,1]$, then for all $x\in \Lambda$ and $\tau>0$ with $f_\tau$ transitive, we have
%\[
%C_0^{-\tau} m_x^{\tau,r} \leq \mmu_x \leq m_x^{1,\rho_1}
%\quad\Rightarrow\quad
%m_x^{\tau,r} \leq C_0^\tau m_x^{1,\rho_1}.
%\]
%In particular, if $\rho_1$ is such that $d(y,z)<\rho_1$ implies $d(f_ty,f_tz) < \rho$ for all $t\in [0,1]$ (note that we use $\rho$ and not $r$ here), then for all $x\in \Lambda$ and $\tau>0$ with $f_\tau$ transitive, we have
%\begin{equation}\label{eqn:leqK}
%m_x^{\tau,\rho}(W) \leq C_0 m_x^{1,\rho_1}(W) =: K.
%\end{equation}
%Observe that $K$ depends on $\rho$ (since $\rho_1$ does) but not on $\tau$ or $x$.

\subsubsection{Averaged pushforwards converge to equilibrium}\label{sec:push-pf}

As shown above, there is a residual set of $\tau\in (0,\infty)$ such that the discrete-time system $(\Lambda,f_\tau,\psi = \int_0^\tau \ph\circ f_s \,ds)$ satisfies the hypotheses of \cite{CPZ2}. By \cite[Theorem 4.7]{CPZ2}, the measures $m_x^{\tau,r}$ defined in \eqref{eqn:mxC} have the property that 
\begin{equation}\label{eqn:discrete-cvg}
\nu_n^{\tau,r} := \frac 1n \sum_{k=0}^{n-1} (f_{k\tau})_* m_x^{\tau,r}
\xrightarrow{\text{weak*}} m_x^{\tau,r}(W) \mu,
\end{equation}
where $\mu$ is the unique equilibrium measure for $(\Lambda,f_\tau,\psi)$. Moreover, $\mu$ is ergodic, fully supported on $\Lambda$, and has the Gibbs property.

Every equilibrium measure for the continuous-time system $(\Lambda,F,\ph)$ is also an equilibrium measure for the discrete-time system. A measure that is ergodic, fully supported, and Gibbs for the discrete-time system has the same properties for the continuous-time system. Thus to prove Theorem \ref{thm:properties}, it suffices to show that the measures $\nu_t$ defined in \eqref{eqn:nut} converge to the same limit as the measures $\nu_n^{\tau,r}$, up to a scalar. Equivalently, we must prove that if $\mu_1 = \lim_{j\to\infty} \nu_{t_j}$ is any limit point of the measures $\nu_t$, then $\mu_1$ is a scalar multiple of $\mu$. Since $\mu$ is ergodic and $\mu_1$ is invariant, it will suffice to prove that $\mu_1\ll\mu$.

We will do this using \eqref{eqn:mCmu} and the observation that if $m_n,\mu_n$ are any sequences of measures on a compact metric space that are weak* convergent to $m_0,\mu_0$ respectively, and that have the property $m_n(Z) \leq \mu_n(Z)$ for all measurable $Z$, then the same is true of $m_0, \mu_0$. Indeed, this property is equivalent to $\int \zeta \,dm_n \leq \int \zeta \,d\mu_n$ for all nonnegative continuous functions $\zeta$, and this inequality survives the limit by the definition of weak* convergence. We summarize this as
\begin{equation}\label{eqn:leq-limit}
m_n \xrightarrow{\text{weak*}} m_0,
\ \mu_n \xrightarrow{\text{weak*}} \mu_0,
\ m_n \leq \mu_n
\quad\Rightarrow\quad
m_0 \leq \mu_0.
\end{equation}
To prove that $\mu_1 = \lim \nu_{t_j} \ll \mu$, start by observing that given any $t,\eta>0$ and $Z\subset \Lambda$, we have
\begin{align*}
|\nu_{t+\eta}(Z) &- \nu_t(Z)| = \Big| \frac 1{t+\eta} \int_0^{t+\eta} (f_s)_* \mmu_x(Z) \,ds - \frac 1t \int_0^t (f_s)_* \mmu_x(Z) \,ds \Big| \\
&\leq \frac 1{t+\eta} \int_t^{t+\eta} (f_s)_* \mmu_x(Z) \,ds 
+ \Big| \frac 1{t+\eta} - \frac 1t \Big| \int_0^t (f_s)_* \mmu_x(Z) \,ds \\
&\leq \frac{1}{t+\eta} \cdot \eta \mmu_x(W) + \frac{\eta}{t(t+\eta)} \cdot t\mmu_x(W)
= \frac{2\eta\mmu_x(W)}{t+\eta}.
\end{align*}
Thus for every $\zeta \in C(\Lambda)$, any $\tau>0$, and any $t \in [(n-1)t\tau, n\tau]$ for some $n\in\NN$, we can apply this with $\eta = n\tau - t$ (and thus $\frac \eta{t+\eta} \leq \frac{\tau}{n\tau} = \frac 1n$) to get
\begin{equation}\label{eqn:int-zeta}
\Big| \int \zeta \,d\nu_{t} - \int \zeta \,d\nu_{n\tau} \Big|
\leq \frac 2n \|\zeta\| \mmu_x(W).
\end{equation}
This implies that for every $\tau>0$, the sequence $n_j = \lfloor t_j/\tau \rfloor \to\infty$ has the property that $\nu_{n_j\tau} \to \mu_1$.
Now observe that 
\[
\nu_{n\tau} = \frac 1{n\tau} 
\int_0^\tau \sum_{k=0}^{n-1} (f_{k\tau+s})_* \mmu_x \,ds
= \frac 1\tau \int_0^\tau (f_s)_* \Big( \frac 1n \sum_{k=0}^{n-1} (f_{k\tau})_* \mmu_x \Big) \,ds.
\]
Fix $\rho>0$ small enough that \eqref{eqn:rho} is satisfied with $\tau=1$, and thus for all $\tau\in (0,1]$ as well.
By \eqref{eqn:mCmu}, we have
\begin{equation}\label{eqn:nunu}
%C_0^{-\tau} \Big( \frac 1\tau \int_0^\tau (f_s)_* \nu_n^{\tau,r}\,ds \Big) \leq
\nu_{n\tau} \leq
\frac 1\tau \int_0^\tau (f_s)_* \nu_n^{\tau,\rho}\,ds
=: \bar\nu_n^{\tau,\rho}.
\end{equation}
Let $C_3 = C_3(\rho)$ be as in \eqref{eqn:leqK} (note that we replace $r$ by $\rho$) and consider the space $\mathcal{M}$ of all finite Borel measures on $\Lambda$ with total weight $\leq C_3$; this space is compact in the weak* topology, which is induced by some metric $D$. The flow $\{(f_t)_*\}_{t\in\RR}$ is continuous on $\mathcal{M}$, hence uniformly continuous, and thus for every $\ell\in\NN$ there is $\tau\in (0,1)$ such that $D((f_t)_* \nu, \nu) \leq 1/\ell$ whenever $|t|\leq\tau$ and $f_\tau$ is topologically transitive (by Lemma \ref{lem:map-transitive}).

Let $c := m_x^{\tau,\rho}(W)$; by \eqref{eqn:leqK} this is at most $C_3$, and thus $\nu_n^{\tau,\rho} \in \mathcal{M}$ for all $n$. Observe that $\nu_n^{\tau,\rho} \to c \mu$ by \eqref{eqn:discrete-cvg}. Thus there is $N$ such that $D(\nu_n^{\tau,\rho},c\mu) \leq 1/\ell$ for all $n\geq N$. By our choice of $\tau$, we conclude that $D(\bar\nu_n^{\tau,\rho},c\mu) \leq 2/\ell$ for all such $n$. By passing to a subsequence if necessary, we can assume that
\[
\nu_{n_j\tau} \to \mu_1 \text{ and } 
\bar\nu_{n_j}^{\tau,\rho} \to \bar\nu_\ell \text{ with }
D(\bar\nu, c\mu) \leq 2/\ell.
\] 
Then \eqref{eqn:leq-limit} and \eqref{eqn:nunu} give $\mu_1 \leq \bar\nu_\ell$, and since $\bar\nu_\ell\to c\mu$ as $\ell\to\infty$, applying \eqref{eqn:leq-limit} again gives $\mu_1\leq c\mu$. In particular, $\mu_1 \ll \mu$, so ergodicity of $\mu$ and flow-invariance of $\mu_1$ imply that $\mu_1$ is a scalar multiple of $\mu$, which completes the proof.

\subsection{Conformality and continuity: proof of Theorem \ref{thm:conf-cts}}\label{sec:hol-pf}

\subsubsection{Scaling under the dynamics}\label{sec:scaling}

In this section we prove that the system $\{\mmu_W : W\in \WWW^u\}$ is $\ph$-conformal in the sense of \eqref{eqn:pull-scale} from Definition \ref{def:conformal}, using the following.

\begin{lemma}\label{lem:near-c}
Given $W\in \WWW^u$, suppose that $Z\subset W\cap \Lambda$, $t,c\in \RR$, and $\eps>0$ are such that for all $y\in Z$, we have $|\Phi(y,t) - c| < \eps$. Then
\begin{equation}\label{eqn:near-c}
%\mmu_W(Z) = e^{\pm 2\eps} \int_{f_t Z} e^{\Phi(y,t) - tP(\ph)} \,d\mmu_{f_t(W)}(y).
\mmu_{f_t(W)}(f_t Z) = e^{\pm 3\eps} \int_Z e^{tP(\ph) - \Phi(z,t)} \,d\mmu_W(z).
\end{equation}
\end{lemma}

Once the lemma is established, we can deduce \eqref{eqn:pull-scale} as follows: fixing $\eps>0$, write $Z = \bigsqcup_i Z_i$ where each $Z_i$ satisfies the condition of the lemma with some $c_i$, and thus satisfies \eqref{eqn:near-c}. Summing over $i$ gives \eqref{eqn:near-c} for $Z$ as well, 
and since $\eps>0$ was arbitrary this proves \eqref{eqn:pull-scale}. Thus it suffices to prove the lemma.

\begin{proof}[Proof of Lemma \ref{lem:near-c}]
Throughout the following, we fix $Z,W,t,c,\eps$ as in the statement.
The function $x\mapsto \Phi(x,t)$ is continuous on $\Lambda$, hence uniformly continuous, so there is $\delta>0$ such that given any $x\in Z$ and $y\in B(x,\delta) \cap \Lambda$, we have
\[
\Phi(y,t) = \Phi(x,t) \pm \eps = c\pm 2\eps.
\]
By \eqref{eqn:diam-t-eta}, there is $T_0$ such that for all $\tau\geq T_0$ and $x\in Z$, we have $B_\tau(x,r,W) \subset B(x,\delta,W)$, and thus
\begin{equation}\label{eqn:Phi-yt}
\Phi(y,t) = c\pm 2\eps \text{ for every $y$ such that } B_\tau(y,r,W) \cap Z \neq \emptyset.
\end{equation}
Let $\EE'(Z,T) = \{ \EEE \subset \EE(Z,T) : B_\tau(y,r,W) \cap Z \neq\emptyset$ for all $(y,\tau) \in \EEE \}$, and observe that in the definition of $\mmu_W$ it suffices to take an infimum over $\EE'(Z,T)$. Now consider the map
\begin{align*}
S\colon (W\cap\Lambda)\times [T,\infty) &\to (f_t(W) \cap \Lambda)\times [T-t,\infty), \\
(y,\tau) &\mapsto (f_t y, \tau-t).
\end{align*}
Given $t_1$ as in Lemma \ref{lem:trans-Bowen} and $\tau\geq \max(t_1,t_1+t)$, that lemma gives
\begin{align*}
f_t(B_\tau(y,r,W))
&= f_t(f_{-\tau}(f_\tau y, r, f_\tau W)) \\
&= f_{-(\tau-t)}(f_{\tau-t}(f_t y), r, f_{\tau-t}(f_t W))
= B_{\tau-t}(f_t y,r,f_t W).
\end{align*}
Thus for $T \geq \max(t_1,t_1+t)$, we see that $\EEE \subset (W\cap \Lambda) \times [T,\infty)$ lies in $\EE'(Z,T)$ if and only if $S(\EEE)$ lies in $\EE'(f_t Z, T-t)$; in other words, $S$ gives a bijection between $\EE'(Z,T)$ and $\EE'(f_t Z, T-t)$.
Moreover, \eqref{eqn:Phi-yt} gives the following estimate for each $(y,\tau) \in \EE'(Z,T)$:
\[
%\int_0^t \ph(f_s y) \,ds = c \pm \eps + \int_0^{\tau - t} \ph(f_s(f_t y)) \,ds,
\Phi(y,\tau) = 
\Phi(y,t) + \Phi(f_t y, \tau - t) =
c \pm 2\eps + \Phi(f_t y, \tau - t).
\]
Now we conclude that
\begin{align*}
\mmu_W(Z) &= \lim_{T\to\infty} \inf_{\EEE \in \EE'(Z,T)} \sum_{(y,\tau) \in \EEE} 
e^{\Phi(y,\tau) - \tau P(\ph)} \\
&= e^{c\pm 2\eps} \lim_{T\to\infty} \inf_{\EEE \in \EE'(Z,T)} \sum_{(y,\tau) \in \EEE}
e^{\Phi(f_t y, \tau -t) - \tau P(\ph)} \\
&= e^{c\pm 2\eps} \lim_{T\to\infty} \inf_{\EEE' \in \EE'(f_t Z,T-t)} \sum_{(y',\tau') \in \EEE}
e^{\Phi(y',\tau') - (\tau' + t)P(\ph)},
\end{align*}
which yields
\begin{equation}\label{eqn:scale-0}
\mmu_W(Z) = e^{c\pm 2\eps - tP(\ph)} \mmu_{f_t (W)}(f_t Z).
\end{equation}
At the same time, we have $\Phi(z,t) = c\pm \eps$ on $Z$, so
\begin{equation}\label{eqn:scale-1}
\int_{Z} e^{tP(\ph) - \Phi(z,t)} \,d\mmu_{W}(z)
= e^{tP(\ph) - c\pm \eps} \mmu_{W}(Z).
\end{equation}
%Comparing \eqref{eqn:scale-0} and \eqref{eqn:scale-1}, we see that both of their left-hand sides lie in the interval $[C_2e^{-\eps}, C_2 e^{\eps}]$, where $C_2 = e^{c-tP(\ph)} \mmu_{f_t (W)}(f_t Z)$. It follows that their ratio lies in $[e^{-2\eps}, e^{2\eps}]$, which proves \eqref{eqn:near-c}.
Combining \eqref{eqn:scale-0} and \eqref{eqn:scale-1} gives \eqref{eqn:near-c} and proves the lemma.
\end{proof}

\subsubsection{Continuity}\label{sec:cts}

In this section we prove that the system $\{\mmu_W : W\in \WWW^u\}$ is continuous in the sense of Definition \ref{def:cts}.
Let $W_1,W_2,\pi$ be as in Definition \ref{def:cts}.
Given $t\geq 0$, 
there is a weak-stable $\delta$-holonomy $f_t W_1 \cap \Lambda \to f_t W_2 \cap \Lambda$ given by 
\begin{equation}\label{eqn:pit}
\pi_t = f_t \circ \pi \circ f_{-t}.
\end{equation}

\begin{lemma}\label{lem:stretch}
For every $\eps>0$, there is $\delta>0$ such that if $W_1,W_2,\pi$ are as in Definition \ref{def:cts}, then for all $t\geq 0$ and $x,y\in W_1\cap \Lambda$ such that $d(f_t x, f_t y) \leq r_0$, we have
\begin{equation}\label{eqn:stretch}
%e^{-\delta} \leq \frac{d(\pi_t (f_t x), \pi_t (f_t y))}{d(f_t x, f_t y)} \leq e^\delta
d_{f_t W_2}(\pi_t (f_t x), \pi_t(f_t y)) = e^{\pm \eps} d_{f_t W_1}(f_t x, f_t y).
\end{equation}
\end{lemma}
\begin{proof}
Choose $\delta>0$ such that \eqref{eqn:stretch} holds when $t=0$ for every pair of admissibles $W_1,W_2$ that are related by a $\delta$-holonomy and satisfy $d_{C^1}(W_1,W_2) < \delta$. Then by cone-invariance, $B(f_t x, r_0, f_t W_1)$ and its image under $\pi_t$ are admissibles containing $f_t x,f_t y$ and $\pi_t(f_tx), \pi_t(f_ty)$, respectively. Moreover, the $C^1$ distance between these admissibles is nonincreasing in $t$ due to uniform expansion along $\Eu$ and nonexpansion along $\Es\oplus E^0$, and thus is at most $\delta$ for every $t\geq 0$, so \eqref{eqn:stretch} continues to hold for all such $t$.
\end{proof}

\begin{lemma}\label{lem:time-shift}
For every $r\in (0,r_0)$ and $\eta>0$, there exists $\delta>0$ such that if $W_1,W_2,\pi$ are as in Definition \ref{def:cts},
then for every $x\in W_1 \cap \Lambda$ and $t\geq \eta$, we have
\begin{equation}\label{eqn:ball-in-ball}
\pi(B_t(x,r,W_1)\cap \Lambda) \subset B_{t-\eta}(\pi x, r, W_2).
\end{equation}
\end{lemma}
\begin{proof}
Choose $\eps>0$ such that $r e^{\eps} < \min(r_0, r\lambda^{-\eta}) =: r'$, and let $\delta>0$ be given by Lemma \ref{lem:stretch}. Then given $W_1,W_2,\pi,x,t$ as in the statement, it follows from Lemma \ref{lem:stretch} that for every $y\in B_t(x,r,W_1) \cap \Lambda$ and $\tau \in [0,t]$, we have
\[
d_{f_\tau W_2}(f_\tau (\pi x), f_\tau(\pi y))
\leq e^\eps d_{f_\tau W_1} (f_\tau x, f_\tau y) < e^\eps r < r'.
\]
Since this holds for all $\tau \in [0,t]$, we have $\pi y \in B_t(\pi x, r', W_2)$, and we conclude that
\begin{equation}\label{eqn:piBt}
\pi(B_t(x,r,W_1)\cap\Lambda) \subset B_t(\pi x, r', W_2)
\subset B_{t-\eta}(\pi x, r' \lambda^\eta, W_2),
\end{equation}
where the last inclusion uses \eqref{eqn:W-Bowen-2}. Since $r'\lambda^\eta \leq r$, this proves the lemma.
\end{proof}

Now we prove continuity. Given $\eps>0$, choose $\eta>0$ such that
\begin{equation}\label{eqn:eta}
3\eta \|\ph\| + \frac{|\ph|_\sigma \eta^\sigma}{\sigma |\log \lambda|} \leq \frac \eps2
\quad\text{and}\quad
\eta |P(\ph)| \leq \frac \eps2.
\end{equation}
Then choose $\delta \in (0,\eta]$ such that Lemma \ref{lem:time-shift} holds. Given $W_1,W_2,\pi$ as in Definition \ref{def:cts} for this choice of $\delta$,
it follows from Lemma \ref{lem:time-shift} that if $\EEE\in \EE(Z,T)$ for some $Z\subset W_1\cap \Lambda$, then
\[
\pi(Z)
\subset \bigcup_{(x,t) \in \EEE} \pi (B_t(x,r,W_1)\cap \Lambda)
\subset \bigcup_{(x,t) \in \EEE} B_{t-\eta}(\pi x, r, W_2).
\]
From this we see that $\{(\pi x, t-\eta) : (x,t) \in \EEE\} \subset \EE(\pi Z, T-\eta)$, and we deduce that
\begin{equation}\label{eqn:mW2leq}
\mmu_{W_2}(\pi Z) \leq \lim_{T\to\infty} \inf_{\EEE\in \EE(Z,T)}
\sum_{(x,t)\in\EEE} e^{\Phi(\pi x, t-\eta) - (t-\eta)P(\ph)}.
\end{equation}
To proceed we must compare $\Phi(\pi x, t-\eta)$ and $\Phi(x,t)$. There is $\tau \in (-\delta,\delta)$ such that $f_t(x) \in \Ws(\pi x,\delta)$, and thus for all $\hat t>0$ we have
\begin{equation}\label{eqn:PhiPhi}
\begin{aligned}
|\Phi( f_t x, \hat t) - \Phi(\pi x, \hat t)|
&\leq \int_0^\infty |\ph(f_s (f_t x)) - \ph(f_s (\pi x))|\,ds \\
&\leq \int_0^\infty |\ph|_\sigma (\delta \lambda^s)^\sigma \,ds
= \frac{|\ph|_\sigma \delta^\sigma}{\sigma |\log \lambda|}.
\end{aligned}
\end{equation}
Since $\delta \leq \eta$, we deduce that $|\tau - \eta| \leq 2\eta$, so
\begin{align*}
\Phi(\pi x, t-\eta) &= \Phi(\pi x, t-\tau) \pm 2\eta \|\ph\|, \\
\Phi(x,t) &= \Phi(f_\tau x, t-\tau) \pm \eta \|\ph\|,
\end{align*}
which together with \eqref{eqn:eta} and \eqref{eqn:PhiPhi} gives
\[
|\Phi(\pi x, t-\eta) - \Phi(x,t)| \leq \eps/2.
\]
Using this estimate in \eqref{eqn:mW2leq} gives
\[
\mmu_{W_2}(\pi Z) \leq \lim_{T\to\infty} \inf_{\EEE\in \EE(Z,T)} 
e^{\Phi(x,t) + \eps/2 - tP(\ph)} e^{\eta P(\ph)} 
\leq e^{\eps} \mmu_{W_1}(Z),
\]
where the last inequality uses the second half of \eqref{eqn:eta}. By symmetry, this proves that $\mmu_{W_2}(\pi Z) = e^{\pm \eps} \mmu_{W_1}(Z)$, which completes the proof of Theorem \ref{thm:conf-cts}.

\subsubsection{Alternate definitions for discrete-time and non-adapted metrics}\label{sec:alternate}

In discrete time, the system of measures $\{\mmu_x\}_{x\in \Lambda}$ was introduced in \cite{CPZ,CPZ2}, which also proved that this system is $\ph$-conformal and has an absolute continuity property under stable holonomies. However, it remains open whether this construction produces a continuous system in the sense of Definition \ref{def:cts}; see \cite[Theorem 4.10, \S6.2.2]{CPZ} and 
\cite[Theorem 4.6, \S7.3]{CPZ2} for the results on holonomy maps that take the place of this property, and note that no specific formula for the Radon--Nikodym derivative is given there (compare to Theorem \ref{thm:cs-hol}, which we prove in the next section).

The key step in the proof of continuity in \S\ref{sec:cts} is to observe that under a $\delta$-holonomy, Bowen balls of radius $r$ on $W_1$ are mapped into Bowen balls of a slightly larger radius $r'$ on $W_2$, and that these can in turn be covered by Bowen balls of the original radius by making a small increase in the continuous time parameter $t\in \RR^+$ and then using \eqref{eqn:W-Bowen-2}. In discrete time, this step does not work because there is no way to make an arbitrarily small increase in the time parameter $n\in \NN$.

One way to obtain continuity as in Definition \ref{def:cts} for the discrete-time construction would be to allow the radius $r$ to vary, and to introduce an extra factor into the ``weight'' assigned each ball in the C-structure. For example, one might fix a continuous function $\zeta(r)$ and then redefine the C-structure on $X=\Wu(x,\delta)\cap \Lambda$ by
\[
\SSS = X \times \NN \times (0,\infty),
\qquad
U_{(x,n,r)} = B_n(x,r),
\]
with $\xi,\eta,\psi$ given by
\[
\xi(x,n,r) = \zeta(r) e^{\sum_{k=0}^{n-1} \ph(f^k x)},
\qquad
\eta(x,n,r) = \psi(x,n,r) = e^{-n}.
\]
With an appropriate choice of $\zeta$ -- say constant and positive near $0$, then increasing quickly on a small interval of $r$ so as to guarantee that all ``optimal'' enveloping sets use only values of $r$ from this interval -- it should be possible to guarantee that in the discrete-time uniformly hyperbolic setting, the system of leaf measures defined by these C-structures satisfies the continuity property in Definition \ref{def:cts}, and that they are uniformly equivalent to the original leaf measures defined in \cite{CPZ}.

It is also worth pointing out that our proofs of $\ph$-conformality and continuity relied on the fact that we work in an adapted metric. Working in a different metric would produce uniformly equivalent systems of measures, following arguments similar to those in \S\ref{sec:discrete}. It is not clear whether these new systems of measures would satisfy $\ph$-conformality and continuity: the arguments at the end of \cite{bH89} suggest that one might expect ``almost-everywhere'' results instead of ``everywhere'' results in this case.

Because we work in an adapted metric, the leafwise Bowen balls are exactly the metric balls in the Hamenst\"adt--Hasselblatt metric from \cite{uH89,bH89}, and then the main thing differentiating our construction from theirs is that we assign weights according to the potential function.\footnote{Note that \cite[Lemma 6]{bH89} gives the proof of continuity for the zero potential using that construction.} An approach that might perhaps be more natural, and which is done in \cite{CPZ,CPZ2}, is to use the extrinsic metric from $M$ on each leaf, instead of the intrinsic leaf metric; doing this here would once again lead to uniformly equivalent systems of measures (again following the arguments in \S\ref{sec:discrete}) but we would not be able to carry out our proofs of $\ph$-conformality or continuity, since we would not have Lemma \ref{lem:trans-Bowen}.

\subsection{Weak-stable holonomy formula: proof of Theorem \ref{thm:cs-hol}}\label{sec:cs-hol}

To prove Theorem \ref{thm:cs-hol}, we start by using the chain rule for Radon--Nikodym derivatives to reduce to the case of two $u$-admissible manifolds that are extremely close; then we will use the continuity property.

\begin{lemma}\label{lem:push-omega}
Suppose $W_1,W_2 \in \WWW^u$ are related by a weak-stable holonomy $\pi$. Given $s,t\geq 0$, let $\pi_{s,t} = f_t \circ \pi \circ f_{-s}$ be the corresponding weak-stable holonomy between $f_s W_1 \cap \Lambda$ and $f_t W_2 \cap \Lambda$. Suppose moreover that $(\pi_{s,t})_*\bmu_{f_s W_1} \ll \bmu_{f_t W_2}$. Then $\pi_* \bmu_{W_1} \ll \bmu_{W_2}$, and we have
\begin{equation}\label{eqn:push-omega}
\frac{d(\pi_* \bmu_{W_1})}{d\bmu_{W_2}}(\pi z)
= e^{\Phi(z,s) - \Phi(\pi z, t) - (s-t)P(\ph)} \frac{d(\pi_{s,t})_* \bmu_{f_s W_1}}{d\bmu_{f_t W_2}}(\pi_{s,t} f_s z).
\end{equation}
%where $\omega_t(z,y) = \Phi(z,t) - \Phi(y,t)$ as in \eqref{eqn:omega-t}.
\end{lemma}
\begin{proof}
We use the chain rule for Radon--Nikodym derivatives in the following form: given measures $\nu_0,\nu_1,\nu_2,\nu_3$ and maps $g^1,g^2,g^3$ such that $g^i_* \nu_{i-1} \ll \nu_i$ for $i=1,2,3$, and writing $x_i = g_i(x_{i-1})$, we have
\begin{equation}\label{eqn:RN-chain}
\begin{aligned}
\frac{d(g^3_* g^2_* g^1_* \nu_0)}{d\nu_3}(x_3)
&= \frac{d(g^3_* g^2_* g^1_* \nu_0)}{d(g^3_*g^2_*\nu_1)}(x_3)
\frac{d(g^3_*g^2_*\nu_1)}{d(g^3_* \nu_2)}(x_3)
\frac{d(g^3_* \nu_2)}{d\nu_3}(x_3) \\
&= \frac{d(g^1_* \nu_0)}{d\nu_1}(x_1)
\frac{d(g^2_* \nu_1)}{d\nu_2}(x_2)
\frac{d(g^3_* \nu_2)}{d\nu_3}(x_3).
\end{aligned}
\end{equation}
With $g^1 = f_s$, $g^2 = \pi_t$, and $g^3 = f_{-t}$, \eqref{eqn:RN-chain} gives
\[
\frac{d(\pi_* \bmu_{W_1})}{d\bmu_{W_2}}(\pi z)
= \frac{d(f_s)_* \bmu_{W_1}}{d\bmu_{f_s W_1}}(f_s z)
\frac{d(\pi_{s,t})_* \bmu_{f_s W_1}}{d\bmu_{f_t W_2}}(f_t \pi z)
\frac{d(f_{-t})_* \bmu_{f_t W_2}}{d\bmu_{W_2}}(\pi z).
\]
From $\ph$-conformality, the first and third factors simplify as
\[
\frac{d(f_s)_* \bmu_{W_1}}{d\bmu_{f_s W_1}}(f_s z)
\frac{d(f_{-t})_* \bmu_{f_t W_2}}{d\bmu_{W_2}}(\pi z)
= e^{\Phi(z,s) - sP(\ph)} e^{tP(\ph) - \Phi(\pi z,t)},
\]
completing the proof of Lemma \ref{lem:push-omega}.
\end{proof}

Now as in Theorem \ref{thm:cs-hol} we assume that there is a weak-stable $\delta_0$-holonomy $\pi\colon W_1 \cap \Lambda \to W_2 \cap \Lambda$.
Given $x\in W_1\cap \Lambda$, there is $\tau(x) \in (-\delta_0,\delta_0)$ such that $f_{\tau(x)} \in \Ws(\pi x, \delta_0)$. Fix $\eps>0$ and let $\delta>0$ be given by continuity (Definition \ref{def:cts}).  For each $x\in W_1 \cap \Lambda$ there are $t>0$ and a relatively open set $W_1' \subset W_1$ containing $x$ such that writing $W_2' = \pi(W_1')$, we have the following, using Lemma \ref{lem:push-trans}.
\begin{itemize}
\item $f_{t+\tau(x)}(W_1')$ and $f_t(W_2')$ are $u$-admissible.
\item $d_{C^1}(f_{t+\tau(x)}(W_1'), f_t(W_2')) < \delta$.
\item $\pi_{t+\tau,t} \colon f_{t+\tau(x)}(W_1') \cap \Lambda \to f_t(W_2') \cap \Lambda$ is a weak-stable $\delta$-holonomy.
\item $|\tau(z) - \tau(x)| (\|\ph\| +|P(\ph)|)< \eps$ for all $z\in W_1'$.
\item $\Phi(z,t+\tau(z)) - \Phi(\pi z, t) - \tau(z) P(\ph) = \omega^+(z,\pi z) \pm \eps$ for all $z\in W_1'$.
\end{itemize}
Using the first three properties, continuity, and Lemma \ref{lem:push-omega} (with $s = t+\tau(x)$) gives
\[
\frac{d(\pi_* \bmu_{W_1'})}{d\bmu_{W_2'}}(\pi z) = e^{\Phi(z,t+\tau(x)) - \Phi(\pi z, t) - \tau(x) P(\ph) \pm \eps}.
\]
Then the fourth property gives
\[
\Phi(z,t+\tau(x)) - \tau(x) P(\ph) = P(z,t+\tau(z)) - \tau(z) P(\ph) \pm \eps,
\]
from which we deduce that
\[
\frac{d(\pi_* \bmu_{W_1'})}{d\bmu_{W_2'}}(\pi z) = e^{\Phi(z,t+\tau(z)) - \Phi(\pi z, t) - \tau(z) P(\ph) \pm 2\eps}
= e^{\omega^+(z,\pi z) \pm 3\eps},
\]
where the last equality uses the fifth property above. Since $\bmu_{W_j'}$ and $\bmu_{W_j}$ agree on $W_j'$, we get the same result with $W_j'$ replaced by $W_j$. Finally, since $\eps>0$ was arbitrary, this proves \eqref{eqn:cs-hol}.

\subsection{Product construction: proof of Theorem \ref{thm:product}}

\subsubsection{Proof of the theorem}

In this section we prove Theorem \ref{thm:product}. There are four things to prove.

\begin{enumerate}
\item The equations \eqref{eqn:muq}--\eqref{eqn:muq4} all agree.
\item There is a unique measure $\mu$ such that $\mu|_{R_q} = \mu_q$ for all $q\in \Lambda$.
\item The measure $\mu$ is flow-invariant, nonzero, and finite.
\item $\mu$ is a scalar multiple of the unique equilibrium measure.
\end{enumerate}

For the first, start by defining a function
$h_q\colon \Ru_q\times \Rcs_q \to (0,\infty)$ by
\begin{equation}\label{eqn:hq}
h_q(x,y) := e^{\omega^+([x,y],x) + \omega^-([x,y],y)}.
\end{equation}
Then recalling \eqref{eqn:xyz}, we can write \eqref{eqn:muq} as
\[
\mu_q(Z) = \int_Z h_q(\Psi_q^{-1}(z)) \,d((\Psi_q)_*(\bmu_q\times\bmcs_q))(z),
\]
while \eqref{eqn:muq2} becomes
\[
\mu_q(Z) = \int_{\Psi_q^{-1}(Z)} h_q(x,y) \,d(\bmu_q\times\bmcs_q)(x,y).
\]
These integrals agree by the definition of pushforward, so \eqref{eqn:muq} and \eqref{eqn:muq2} agree. To compare \eqref{eqn:muq2} and \eqref{eqn:muq3} it suffices to observe that given any $y\in \Rcs_q$, there is a weak-stable $\delta$-holonomy $\pi\colon \Ru_q \to \Ru_q(y)$ given by $\pi(x) = [x,y]$, and Theorem \ref{thm:cs-hol} shows that for every continuous $\zeta \colon \Ru_q(y) \to \RR$, we have
\begin{equation}\label{eqn:vary-leaf}
\begin{aligned}
\int_{\Ru_q} \zeta(\pi x) \mathbf{1}_Z([x,y]) \,d\bmu_q(x)
&= \int_{\Ru_q(y) \cap Z} \zeta(z) \,d(\pi_* \bmu_q)(z) \\
&= \int_{\Ru_q(y) \cap Z} \zeta(z) e^{\omega^+([z,q],z)} \,d\bmu_y(z)
\end{aligned}
\end{equation}
We use this with 
\[
\zeta(z) = h_q(\Psi_q^{-1}(z)) = e^{\omega^+(z,[z,q]) + \omega^-(z,y)}
\]
so that
\[
\zeta(\pi x) = e^{\omega^+([x,y],x) + \omega^-([x,y],y)}
\quad\text{and}\quad
\zeta(z) e^{\omega^+([z,q],z)} = e^{\omega^-(z,y)}.
\]
Then the first integral in \eqref{eqn:vary-leaf} is the inner integral in \eqref{eqn:muq2}, while the last integral in \eqref{eqn:vary-leaf} becomes the inner integral in \eqref{eqn:muq3},
so that \eqref{eqn:muq2} and \eqref{eqn:muq3} agree. A similar argument shows the equivalence of \eqref{eqn:muq2} and \eqref{eqn:muq4}.

%$h_q \colon R_q \to (0,\infty)$ by
%\begin{equation}\label{eqn:hq}
%h_q(z) = e^{\omega^+(z,[z,q]) + \omega^-(z,[q,z])}.
%\end{equation}

To prove that the family of measures $\{ \mu_q : q\in \Lambda \}$ defines a unique Borel measure $\mu$ on $\Lambda$, it suffices to show that they are consistent in the sense that
\begin{equation}\label{eqn:consistent}
\mu_q(Z) = \mu_p(Z) \text{ whenever }  Z\subset R_q \cap R_p;
\end{equation}
then one defines $\mu$ by $\mu(Z) = \mu_q(Z)$ whenever $Z\subset R_q$, and extends by additivity.\footnote{See \cite[\S4]{gM70} and \cite[\S4]{gM04} for more details of this standard procedure.}
We observe that when $p \in \Rcs_q$, \eqref{eqn:consistent} follows immediately from \eqref{eqn:muq3} and the fact that $\bmcs_q$ and $\bmcs_p$ agree on $\Rcs_q \cap \Rcs_p$. Similarly, \eqref{eqn:muq4} establishes \eqref{eqn:consistent} when $p \in \Ru_q$. 

Now given any $q,p\in \Lambda$ such that $R_q \cap R_p \neq \emptyset$, we can take $x$ to be any point in the intersection and then observe that the measures $\mu_q$, $\mu_{[x,q]}$, $\mu_x$, $\mu_{[x,p]}$, $\mu_p$ all agree on $R_q \cap R_p$ by the previous paragraph. This proves \eqref{eqn:consistent} and thus establishes existence of a unique $\mu$ as claimed.

The fact that $\mu$ is nonzero and finite follows from the corresponding properties for each $\mu_q$, together with the fact that $\Lambda$ can be covered with finitely many sets $R_q$ by compactness. For flow-invariance, it suffices to show that given $q\in \Lambda$, $Z\subset R_q$, and $t\in \RR$ such that $f_tZ \subset R_q$, we have $\mu_q(f_t Z) = \mu_q(Z)$.
To show this, we observe that on the leaf $\Wcs(x)$, 
we can use \eqref{eqn:pull-scale-ms} and \eqref{eqn:weak-measures}, together with the observation that $-tP(\ph) + \Phi(z,t) = \omega^+(z,f_t z)$, to write
\begin{align*}
\bmcs_{f_t x}(f_t Z) &= \int_{-\delta}^\delta \bms_{f_\tau(f_t x)} (f_t Z) \,d\tau 
=\int_{-\delta}^\delta \bms_{f_t(f_\tau x)} (f_t Z) \,d\tau \\
&= \int_{-\delta}^\delta \int_Z e^{\omega^+(z,f_t z)} \,d\bms_{f_\tau x}(z) \,d\tau
= \int_Z e^{\omega^+(z,f_t z)} \,d\bmcs_x(z) \,d\tau.
\end{align*}
We will use this with $t$ replaced by $-t$ to obtain
\[
\frac{d((f_{t})_* \bmcs_{f_{-t} x})}{d\bmcs_x}(y) = e^{\omega^+(y,f_{-t} y)}
\quad\text{and}\quad
\frac{d\bmcs_x}{d((f_{t})_* \bmcs_{f_{-t} x})}(y) = e^{\omega^+(f_{-t} y,y)}.
\]
Using the second of these together with the definition of $\mu_q$ in \eqref{eqn:muq4}, we have
\begin{align*}
\mu_q(f_t Z) &= \int_{\Ru_q} \int_{\Rcs_q(x)}
e^{\omega^+(y,x)} \mathbf{1}_{f_t Z}(y) \,d\bmcs_x(y) \,d\bmu_q(x) \\
&= \int_{\Ru_q} \int_{\Rcs_q(x)} e^{\omega^+(y,x)} \mathbf{1}_Z(f_{-t} y)
e^{\omega^+(f_{-t} y, y)} \,d((f_t)_* \bmcs_{f_{-t} x})(y) \,d\bmu_q(x).
\end{align*}
Writing $y' = f_{-t} y$ and using the cocycle relation \eqref{eqn:cocycle}, this becomes
\[
\mu_q(f_t Z) =
\int_{\Ru_q} \int_{\Rcs_q(x)} e^{\omega^+(y',x)} \mathbf{1}_Z(y')
\,d\bmcs_{f_{-t} x}(y') \,d\bmu_q(x).
\]
Finally, since for all $y'\in Z$ we also have $f_t(y') \in R_q$ (by our assumption that $f_t Z \subset R_q$), we can use the fact that $\bmcs_{f_{-t}x}$ and $\bmcs_x$ agree on the overlap of the center-stable leaves to obtain
\[
\mu_q(f_t Z) =
\int_{\Ru_q} \int_{\Rcs_q(x)} e^{\omega^+(y',x)} \mathbf{1}_Z(y')
\,d\bmcs_{x}(y') \,d\bmu_q(x)
= \mu_q(Z),
\]
establishing the desired flow-invariance.

Finally, to prove that $\mu$ is a scalar multiple of the unique equilibrium measure, it suffices to show that there is $C_4>0$ such that the upper Gibbs bound $\mu_q(B_t(q,r)) \leq C_4 e^{\Phi(q,t) - tP(\ph)}$ holds for all $q\in \Lambda$ and $t>0$, since then $\mu$ is absolutely continuous with respect to the unique invariant Gibbs probability measure, which is ergodic.

To prove the upper Gibbs bound, observe that for $\delta$ sufficiently small, for each $q\in \Lambda$ and $y\in \Wcs(q,\delta) \cap \Lambda$, writing $W = \Wu(y,L\delta)$ we have
\[
B_t(q,\delta) \cap W \subset B_t(y,2\delta, W) = f^{-t} \Wu(f_t y, 2\delta),
\]
where the last equality uses Lemma \ref{lem:W-Bowen}. Thus $\ph$-conformality gives
\[
\bmu_y(B_t(q,\delta)) \leq \sup_{z\in B_t(q,\delta)} e^{\Phi(z,t) - tP(\ph)} \bmu_{f_t y}(\Wu(f_t y, 2\delta)).
\]
Since $\ph$ is H\"older continuous, it satisfies the Bowen property
\[
\sup_{q\in \Lambda} \sup_{t>0} \sup_{z\in B_t(q,\delta)} |\Phi(z,t) - \Phi(q,t)| < \infty.
\]
Using the previous two estimates and Theorem \ref{thm:properties}, we get
\[
\bmu_y(B_t(q,\delta)) \leq C_5 e^{\Phi(q,t) - tP(\ph)},
\]
where $C_5>0$ does not depend on $q$, $y$, or $t$. Now the upper Gibbs bound for $\mu_q$ follows immediately from \eqref{eqn:muq3} and the fact that $e^{\omega^-(x,y)}$ and $\bmcs_q(\Rcs_q)$ are uniformly bounded above independently of $q\in \Lambda$, $y\in \Ru_q$, and $x\in \Rcs_q(y)$. This completes the proof of Theorem \ref{thm:product}.

\subsubsection{Proof of Corollary \ref{cor:unique}}\label{sec:unique}

To prove Corollary \ref{cor:unique}, we start by observing that given $q\in \Lambda$, it follows from \eqref{eqn:mu-m} that for $\hat\mu_q$-a.e.\ $y\in \Wcs_q$, we have
\begin{equation}\label{eqn:gqy}
\frac{\bmu_y(Z)}{\mmu_y(Z)} = \frac{\bmu_y(\Ru_q(y))}{\mmu_y(\Ru_q(y))} =: g(q,y) \text{ for all Borel } Z \subset \Ru_q(y).
\end{equation}
Observe that $g(q,y)$ is independent of $\delta$ (even though $\Ru_q(y)$ depends on $\delta$) because one obtains the same ratio for every Borel $Z$. Since both systems $\{\bmu_y\}$ and $\{\mmu_y\}$ scale according to Theorem \ref{thm:cs-hol} under weak-stable holonomies, we conclude that $g$ is in fact independent of $y$. More precisely, choosing any $y\in \Wcs_q$ such that \eqref{eqn:gqy} holds, let $\pi\colon \Ru_q \to \Ru_q(y)$ be a weak-stable $\delta_0$-holonomy: then by \eqref{eqn:cs-hol}, for every Borel $Z\subset \Ru_q$ we have
\[
\frac{\bmu_q(Z)}{\mmu_q(Z)} = \frac{\pi_* \bmu_q(\pi Z)}{\pi_* \mmu_q(\pi Z)} = \frac{\int_{\Ru_q(y)} e^{\omega^+(\pi^{-1} z, z)} \,d\bmu_y(z)}
{\int_{\Ru_q(y)} e^{\omega^+(\pi^{-1} z, z)} \,d\mmu_y(z)} = g(q,y) =: g(q).
\]
Now given any $t\in \RR$, since $\bmu_q$ and $\mmu_q$ are fully supported on $\Ru_q$, we can choose a Borel set $Z\subset \Ru_q$ such that $f_t Z \subset \Ru_{f_t q}$, and use the $\ph$-conformality property \eqref{eqn:pull-scale} to deduce that
\[
g(f_t q) = \frac{\bmu_{f_t q}(f_t Z)}{\mmu_{f_t q}(f_t Z)} = 
\frac{\int_Z e^{tP(\ph) - \Phi(z,t)} \,d\bmu_q(z)}
{\int_Z e^{tP(\ph) - \Phi(z,t)} \,d\mmu_q(z)} = g(q).
\]
Thus $g$ is flow-invariant, and since the unique equilibrium measure $\mu$ is ergodic, $g$ is constant $\mu$-a.e.: there is $c>0$ such that $g(q) = c$ for $\mu$-a.e.\ $q\in \Lambda$.  Moreover, since $\supp \mu = \Lambda$ and $g$ is continuous (because the systems $\{\bmu_q\}$ and $\{\mmu_q\}$ are), we conclude that $g(q)=c$ for every $q\in \Lambda$, which completes the proof of Corollary \ref{cor:unique}.

\subsection{Refining in all directions: proof of Theorem \ref{thm:direct}}

Let $m$ be the measure on $\Lambda$ defined by \eqref{eqn:E*}--\eqref{eqn:m}. Observe that it is a Borel measure by Lemma \ref{lem:get-metric} because \eqref{eqn:diam-t-eta} and the local product structure give
\[
\diam B^*_{s,t}(x,r) \leq 
 2rL(\lambda^{s-t_0} + \lambda^{t-t_0}) + \frac {C_6 r}{s+t},
\]
where
\begin{equation}\label{eqn:max-speed}
C_6 := \max_{x\in M} \Big\|\frac d{dt} f_t(x)|_{t=0}\Big\|
\end{equation}
is the maximum speed of the flow, so that $d(x,f_t x) \leq C_6 |t|$ for all $x,t$.

We prove Theorem \ref{thm:direct} by 
showing that $m$ is flow-invariant and has the Gibbs property \eqref{eqn:Gibbs}; we will use the fact that $(\Lambda,F,\ph)$ is already known to have a unique flow-invariant Gibbs probability measure.

\subsubsection{Flow-invariance}\label{sec:flow-inv}

To prove that $m(f_\tau Z) = m(Z)$ for all $Z\subset \Lambda$ and $\tau\in \RR$, we first observe that given $x\in \Lambda$ and $s,t\geq |\tau|$, we have
\[
z\in B^{\pm}_{s,t}(x,r)
\quad\Leftrightarrow\quad
f_\tau z \in B^{\pm}_{s+\tau,t-\tau}(f_\tau x,r).
\]
Moreover, given $z\in B^{\pm}_{s,t}(x,r)$, we have
\[
\beta(f_\tau x, f_\tau z) = \beta(x,z),
\]
from which it follows that
\begin{equation}\label{eqn:flow-B*}
f_\tau B^*_{s,t}(x,r) = B^*_{s+\tau,t-\tau}(f_\tau x, r).
\end{equation}
Now given $Z\subset \Lambda$ and $T\geq |\tau|$, consider the map
\begin{align*}
S\colon \Lambda \times [T,\infty)^2 &\to \Lambda \times [T-|\tau|,\infty)^2, \\
(x,s,t) &\mapsto (f_\tau x, s+\tau, t-\tau).
\end{align*}
By \eqref{eqn:flow-B*}, given $Z\subset \Lambda$ and $\EEE \in \EE^*(Z,T)$, we have
\[
f_\tau Z \subset f_\tau \Big( \bigcup_{(x,s,t) \in \EEE} B^*_{s,t}(x,r) \Big)
\subset \bigcup_{(x,s,t) \in \EEE} B^*_{s+\tau, t-\tau}(f_\tau x, r),
\]
so that $S(\EEE) \subset \EE^*(f_\tau Z,T-|\tau|)$. Since $(s+\tau) + (t-\tau) = s+t$, it follows that
\begin{align*}
m(f_\tau Z) &\leq \lim_{T\to\infty} \inf_{\EEE \in \EE^*(Z,T)} \sum_{(x,s,t) \in \EEE}
\frac {e^{\Phi(f_{-(s+\tau)}(f_\tau x),s+t) - (s+t)P(\ph)}}
{s+t} \\
&= \lim_{T\to\infty} \inf_{\EEE \in \EE^*(Z,T)} \sum_{(x,s,t) \in \EEE}
\frac {e^{\Phi(f_{-s}(x),s+t) - (s+t)P(\ph)}}
{s+t} = m(Z).
\end{align*}
Since $m(f_\tau Z) \leq m(Z)$ for all $\tau \in \RR$ and Borel $Z\subset \Lambda$, we conclude that $m$ is flow-invariant.

\subsubsection{Gibbs property}

We start with a general lemma.

\begin{lemma}\label{lem:B*Bpm}
There are $C_7,T_1,\eta>0$ such that for every flow-invariant measure $\nu$ on $\Lambda$ and every $(x,s,t) \in \Lambda\times [T_1,\infty)^2$, we have
\begin{align}
\label{eqn:B*Bpm}
\nu(B^*_{s,t}(x,r)) &\leq C_7(s+t)^{-1} \nu(B^\pm_{s-\eta,t-\eta}(x,r)) \\
\label{eqn:BpmB*}
\nu(B^\pm_{s,t}(x,r)) &\leq C_7(s+t) \nu(B^*_{s-\eta,t-\eta}(x,r)).
\end{align}
\end{lemma}
\begin{proof}
Given $z\in B^\pm_{s,t}(x,r)\cap \Lambda$, the local product structure gives $y := [x,z] \in \Wu(z,rL)$ and $y' = f_{\beta(x,z)}(y) \in \Ws(x,rL)$ with $f_{-s}(y') \in \Ws(f_{-s} x, rL)$ and $f_t(y) \in \Wu(f_t z, rL)$. Given $\eta \in [0,\min(s,t)]$ and $\tau \in [-s+\eta,t-\eta]$, we deduce from \eqref{eqn:leaves-contract} that\footnote{In fact we are using the analogue of \eqref{eqn:leaves-contract} for the original metric, which is not necessarily adapted; this is the reason for the inclusion of the constant $C_0$.}
\[
d(f_\tau z, f_\tau y) \leq C_0 rL\lambda^\eta
\quad\text{and}\quad
d(f_\tau y', f_\tau x) \leq C_0 rL\lambda^\eta.
\]
Since we also have $d(f_\tau y', f_\tau y) \leq C_6 |\beta(x,z)|$, the triangle inequality gives
\begin{equation}\label{eqn:dbeta}
d(f_\tau x, f_\tau z)
\leq 2C_0 rL \lambda^\eta + C_6 |\beta(x,z)|
\end{equation}
for all $z\in B^\pm_{s,t}(x,r)$ and $\tau \in [-s+\eta,t-\eta]$.

Now we prove \eqref{eqn:B*Bpm}. By \eqref{eqn:dbeta} there are $T_1 > \eta > 0$ such that for all $z\in B^*_{s,t}(x,r)$ and $\tau \in [-s+\eta,t-\eta]$, we have
\[
d(f_\tau x, f_\tau z) \leq 2C_0 rL\lambda^\eta + C_6 T_1^{-1} < r/2.
\]
In other words, we have
\begin{equation}\label{eqn:B*/2}
B^*_{s,t}(x,r) \subset \{ z \in B^\pm_{s-\eta,t-\eta}(x,r/2) : |\beta(x,z)| < r/(s+t) \}.
\end{equation}
Now let $\xi>0$ be such that $d(z,f_\tau z) < r/2$ for all $z\in \Lambda$ and $|\tau| < \xi$; then we have $f_\tau(x) \in B^*_{s,t}(x,r)$ for all $|\tau| < \xi$, and
\begin{equation}\label{eqn:inBpm}
f_\tau(z) \in B^{\pm}_{s-\eta,t-\eta}(x,r)
\text{ for all $z\in B^*_{s,t}(x,r)$ and $|\tau|<\xi$}.
\end{equation}
Consider $(x,s,t) \in \Lambda \times [T_1,\infty)^2$ and let $\gamma = \frac {2r}{s+t}$. Then for all $0\leq k < N := \lfloor \frac \xi\gamma \rfloor = \lfloor \frac \xi{2r}(s+t)\rfloor$, we have $f_{k\gamma} B^*_{s,t}(x,r) \subset B^{\pm}_{s-\eta,t-\eta}(x,r)$ by \eqref{eqn:inBpm}, and moreover these $N$ sets are disjoint because $\beta(x,f_{k\gamma} z) = k\gamma + \beta(x,z) = k\gamma \pm \frac r{s+t}$ for all $z\in B^*_{s,t}(x,r)$. These sets all have the same measure by flow-invariance, and thus
$\nu(B^{\pm}_{s-\eta,t-\eta}(x,r)) \geq N \nu(B^*_{s,t}(x,r))$, which proves \eqref{eqn:B*Bpm}.

Now we prove \eqref{eqn:BpmB*}.
There is $C_8>0$ such that $\beta(x,z) \leq C_8 r$ for all $x\in \Lambda$ and $z\in B(x,r)$. Given $(x,s,t) \in \Lambda\times [T_1,\infty)^2$, let $\gamma = \frac r{s+t}$, $N = \lceil C_8r/\gamma \rceil = \lceil C_8(s+t) \rceil$,  and $x_k = f_{k\gamma}(x)$ for all $|k| \leq N$. Then for every $z\in B^\pm_{s,t}(x,r)$, there is $|k| \leq N$ such that $|\beta(x_k, z)| < \gamma = \frac r{s+t}$. For every $\tau \in [-s+\eta,t-\eta]$, \eqref{eqn:dbeta} gives
\[
d(f_\tau x_k, f_\tau z) \leq 2C_0 rL \lambda^\eta + C_6/(s+t) < r/2,
\]
and thus $z \in B^*_{s-\eta,t-\eta}(x_k,r/2) \subset f_{k\gamma} B^*_{s-\eta,t-\eta}(x,r)$. By flow-invariance we have
\[
\nu(B^\pm_{s,t}(x,r)) \leq \nu \Big( \bigcup_{k=-N}^N f_{k\gamma} B^*_{s-\eta,t-\eta}(x,r) \Big) = (2N+1) \nu(B^*_{s-\eta,t-\eta}(x,r)),
\]
which proves \eqref{eqn:BpmB*}.
\end{proof}

\begin{lemma}\label{lem:*-gibbs}
Let $\mu$ be the unique flow-invariant Gibbs probability measure for $(\Lambda,F,\ph)$. Then there is $C_9>0$ such that for every $(x,s,t) \in \Lambda\times [T_1,\infty)^2$, we have
\begin{equation}\label{eqn:*-gibbs}
\mu(B^*_{s,t}(x,r)) = C_9^{\pm 1} (s+t)^{-1} e^{\Phi(f_{-s} x, s+t) - (s+t)P(\ph)}.
\end{equation}
\end{lemma}
\begin{proof}
Since $\mu$ is a Gibbs measure, there is $Q>0$ such that for all $(x,s,t) \in \Lambda \times [0,\infty)^2$, we have
\begin{multline}\label{eqn:mu-gibbs}
\mu(B^{\pm}_{s,t}(x,r)) = \mu(f_{-s}B^{\pm}_{s,t}(x,r)) = \mu(B_{s+t}(f_{-s} x,r)) \\
=Q^{\pm 1} e^{\Phi(f_{-s} x, s+t) - (s+t)P(\ph)}.
\end{multline}
Combining \eqref{eqn:B*Bpm} and the upper bound in \eqref{eqn:mu-gibbs} gives
\begin{align*}
\mu(B^*_{s,t}(x,r)) &\leq C_7Q (s+t)^{-1} e^{\Phi(f_{-s+\eta} x, s+t-2\eta) - (s+t-2\eta)P(\ph)} \\
&\leq C_7Q (s+t)^{-1} e^{\Phi(f_{-s} x, s+t) - (s+t)P(\ph)} e^{2\eta(\|\ph\| + |P(\ph)|)}
\end{align*}
which proves the upper bound in \eqref{eqn:*-gibbs}. For the lower bound we use \eqref{eqn:BpmB*} and the lower bound in \eqref{eqn:mu-gibbs} to get
\begin{align*}
\mu(B^*_{s,t}(x,r)) &\geq C_7^{-1} (s+t+2\eta)^{-1} \mu(B^\pm_{s+\eta,t+\eta}(x,r)) \\
&\geq \frac{(s+t)e^{\Phi(f_{-s} x, s+t) - (s+t)P(\ph)}}
{ C_7 Q (s+t + 2\eta) (s+t) e^{2\eta(\|\ph\| + |P(\ph)|)}},
\end{align*}
which completes the proof since $(s+t)/(s+t+2\eta) \geq T_1/(T_1+2\eta)$.
\end{proof}

Now we can prove that $m$ has the Gibbs property. For the lower bound, consider  any Borel set $Z\subset \Lambda$, any $T\geq T_1$, and any $\EEE \in \EE^*(Z,T)$: by \eqref{eqn:*-gibbs}, we have
\[
\sum_{(x,s,t) \in \EEE} \frac{e^{\Phi(f_{-s}(x),s+t) - (s+t)P(\ph)}}{s+t}
\geq \sum_{(x,s,t) \in \EEE} \frac{1}{C_9} \mu(B^*_{s,t}(x,r))
\geq \frac{\mu(Z)}{C_9}.
\]
Taking an infimum over all such $\EEE$ and sending $T\to\infty$ gives $m(Z) \geq \mu(Z)/C_9$. Since $\mu$ has the Gibbs property, this proves the lower Gibbs bound for $m$.

To prove the upper Gibbs bound, fix $x\in \Lambda$ and $\tau>0$, and let $Z = B_\tau(x,r/2) \cap \Lambda$.
Given $T>\tau$, let $E_T\subset Z$ be 
maximal with the property that for any $y,z\in E_T$ with $z\neq y$, we have $z\notin B^*_{T,T}(y,r)$;
then $Z\subset \bigcup_{y\in E_T} B^*_{T,T}(y,r)$, while at the same time the sets $B^*_{T,T}(y,r/2)$ are disjoint, and we have
\begin{align*}
m(Z) &\leq \ulim_{T\to\infty} \sum_{y\in E_T} \frac 1{s+t} e^{\Phi(f_{-T}y,2T) - 2TP(\ph)} \\
&\leq \ulim_{T\to\infty} \sum_{y\in E_T} C_9(r/2) \mu(B_{T,T}^*(y,r/2)) \\
&\leq \ulim_{T\to\infty} C_9(r/2) \mu \Big( \bigcup_{y\in E_T} B^*_{T,T}(y,r/2) \Big).
\end{align*}
Given $y\in E_T$, observe that $B^*_{T,T}(y,r/2) \subset B_\tau(y,r/2) \subset B_\tau(x,r)$, and thus the above computation shows that
\[
m(B_\tau(x,r/2)) \leq C_9(r/2) \mu(B_\tau(r)).
\]
Thus the upper Gibbs bound for $m$ (at scale $r/2$) follows from the upper Gibbs bound for $\mu$ (at scale $r$).

\bibliographystyle{amsalpha}
\bibliography{es-formula-ref}

\providecommand{\bysame}{\leavevmode\hbox to3em{\hrulefill}\thinspace}
\providecommand{\MR}{\relax\ifhmode\unskip\space\fi MR }
\providecommand{\MRhref}[2]{%
  \href{http://www.ams.org/mathscinet-getitem?mr=#1}{#2}
}
\providecommand{\href}[2]{#2}
\begin{thebibliography}{BCFT18}

\bibitem[BCFT18]{BCFT}
K.~Burns, V.~Climenhaga, T.~Fisher, and D.~J. Thompson, \emph{Unique
  equilibrium states for geodesic flows in nonpositive curvature}, Geom. Funct.
  Anal. \textbf{28} (2018), no.~5, 1209--1259. \MR{3856792}

\bibitem[BM77]{BM77}
Rufus Bowen and Brian Marcus, \emph{Unique ergodicity for horocycle
  foliations}, Israel J. Math. \textbf{26} (1977), no.~1, 43--67. \MR{451307}

\bibitem[Bow73]{rB73}
Rufus Bowen, \emph{Topological entropy for noncompact sets}, Trans. Amer. Math.
  Soc. \textbf{184} (1973), 125--136. \MR{0338317}

\bibitem[Bow08]{Bow08}
\bysame, \emph{Equilibrium states and the ergodic theory of {A}nosov
  diffeomorphisms}, revised ed., Lecture Notes in Mathematics, vol. 470,
  Springer-Verlag, Berlin, 2008, With a preface by David Ruelle, Edited by
  Jean-Ren\'e Chazottes. \MR{2423393}

\bibitem[Bow75]{rB745}
\bysame, \emph{Some systems with unique equilibrium states}, Math. Systems
  Theory \textbf{8} (1974/75), no.~3, 193--202. \MR{0399413}

\bibitem[BR75]{BR75}
Rufus Bowen and David Ruelle, \emph{The ergodic theory of {A}xiom {A} flows},
  Invent. Math. \textbf{29} (1975), no.~3, 181--202. \MR{0380889}

\bibitem[Cal20]{Cal20}
Benjamin Call, \emph{The {K}-property for some unique equilibrium states in
  flows and homeomorphisms}, 2020, arXiv:2007.00035.

\bibitem[CKW]{CKW2}
Vaughn Climenhaga, Gerhard Knieper, and Khadim War, \emph{Closed geodesics on
  surfaces without conjugate points}, arXiv:2008.02249.

\bibitem[CKW21]{CKW}
Vaughn Climenhaga, Gerhard Knieper, and Khadim War, \emph{Uniqueness of the
  measure of maximal entropy for geodesic flows on certain manifolds without
  conjugate points}, Adv. Math. \textbf{376} (2021), 107452, 44. \MR{4178924}

\bibitem[CLP17]{CLP17}
Vaughn Climenhaga, Stefano Luzzatto, and Yakov Pesin, \emph{The geometric
  approach for constructing {S}inai-{R}uelle-{B}owen measures}, J. Stat. Phys.
  \textbf{166} (2017), no.~3-4, 467--493. \MR{3607577}

\bibitem[Cou16]{yC16}
Yves Coud\`ene, \emph{Ergodic theory and dynamical systems}, Universitext,
  Springer-Verlag London, Ltd., London; EDP Sciences, [Les Ulis], 2016,
  Translated from the 2013 French original [ MR3184308] by Reinie Ern\'{e}.
  \MR{3586310}

\bibitem[CP17]{CP17}
Vaughn Climenhaga and Yakov Pesin, \emph{Building thermodynamics for
  non-uniformly hyperbolic maps}, Arnold Math. J. \textbf{3} (2017), no.~1,
  37--82. \MR{3646530}

\bibitem[CPZ19]{CPZ}
Vaughn Climenhaga, Yakov Pesin, and Agnieszka Zelerowicz, \emph{Equilibrium
  states in dynamical systems via geometric measure theory}, Bull. Amer. Math.
  Soc. (N.S.) \textbf{56} (2019), no.~4, 569--610. \MR{4007162}

\bibitem[CPZ20]{CPZ2}
\bysame, \emph{Equilibrium measures for some partially hyperbolic systems},
  Journal of Modern Dynamics \textbf{16} (2020), 155--205.

\bibitem[CT12]{CT12}
Vaughn Climenhaga and Daniel~J. Thompson, \emph{Intrinsic ergodicity beyond
  specification: {$\beta$}-shifts, {$S$}-gap shifts, and their factors}, Israel
  J. Math. \textbf{192} (2012), no.~2, 785--817. \MR{3009742}

\bibitem[CT16]{CT16}
\bysame, \emph{Unique equilibrium states for flows and homeomorphisms with
  non-uniform structure}, Adv. Math. \textbf{303} (2016), 745--799.
  \MR{3552538}

\bibitem[CT19]{CaT19}
Benjamin Call and Daniel~J. Thompson, \emph{Equilibrium states for products of
  flows and the mixing properties of rank 1 geodesic flows}, 2019,
  arXiv:1906.09315.

\bibitem[CT21]{CT21}
Vaughn Climenhaga and Daniel~J.\ Thompson, \emph{Beyond {B}owen's specification
  property}, Thermodynamic {Formalism: CIRM Jean-Morlet Chair, Fall} 2019 (Mark
  Pollicott and Sandro Vaienti, eds.), Lecture Notes in Mathematics, vol. 2290,
  Springer, 2021, arXiv:2009.09256.

\bibitem[Fed69]{hF69}
Herbert Federer, \emph{Geometric measure theory}, Die Grundlehren der
  mathematischen Wissenschaften, Band 153, Springer-Verlag New York Inc., New
  York, 1969. \MR{0257325}

\bibitem[FH19]{FH19}
Todd Fisher and Boris Hasselblatt, \emph{Hyperbolic flows}, Zurich Lectures in
  Advanced Mathematics, vol.~25, European Mathematical Society, 2019.

\bibitem[Ham89]{uH89}
Ursula Hamenst\"{a}dt, \emph{A new description of the {B}owen-{M}argulis
  measure}, Ergodic Theory Dynam. Systems \textbf{9} (1989), no.~3, 455--464.
  \MR{1016663}

\bibitem[Has89]{bH89}
Boris Hasselblatt, \emph{A new construction of the {M}argulis measure for
  {A}nosov flows}, Ergodic Theory Dynam. Systems \textbf{9} (1989), no.~3,
  465--468. \MR{1016664}

\bibitem[Hay94]{nH94}
Nicolai T.~A. Haydn, \emph{Canonical product structure of equilibrium states},
  Random Comput. Dynam. \textbf{2} (1994), no.~1, 79--96. \MR{1265227}

\bibitem[Kai90]{vK90}
Vadim~A. Kaimanovich, \emph{Invariant measures of the geodesic flow and
  measures at infinity on negatively curved manifolds}, Ann. Inst. H.
  Poincar\'{e} Phys. Th\'{e}or. \textbf{53} (1990), no.~4, 361--393, Hyperbolic
  behaviour of dynamical systems (Paris, 1990). \MR{1096098}

\bibitem[Lep00]{Lep}
Renaud Leplaideur, \emph{Local product structure for equilibrium states},
  Trans. Amer. Math. Soc. \textbf{352} (2000), no.~4, 1889--1912. \MR{1661262}

\bibitem[Mar70]{gM70}
G.~A. Margulis, \emph{Certain measures that are connected with {U}-flows on
  compact manifolds}, Funkcional. Anal. i Prilo\v{z}en. \textbf{4} (1970),
  no.~1, 62--76. \MR{0272984}

\bibitem[Mar04]{gM04}
Grigoriy~A. Margulis, \emph{On some aspects of the theory of {A}nosov systems},
  Springer Monographs in Mathematics, Springer-Verlag, Berlin, 2004, With a
  survey by Richard Sharp: Periodic orbits of hyperbolic flows, Translated from
  the Russian by Valentina Vladimirovna Szulikowska. \MR{2035655}

\bibitem[Pes97]{pes97}
Yakov~B. Pesin, \emph{Dimension theory in dynamical systems}, Chicago Lectures
  in Mathematics, University of Chicago Press, Chicago, IL, 1997, Contemporary
  views and applications. \MR{1489237}

\bibitem[PP84]{PP84}
Ya.~B. Pesin and B.~S. Pitskel', \emph{Topological pressure and the variational
  principle for noncompact sets}, Funktsional. Anal. i Prilozhen. \textbf{18}
  (1984), no.~4, 50--63, 96. \MR{775933}

\bibitem[PS82]{PS82}
Ya.~B. Pesin and Ya.~G. Sina\u{\i}, \emph{Gibbs measures for partially
  hyperbolic attractors}, Ergodic Theory Dynam. Systems \textbf{2} (1982),
  no.~3-4, 417--438 (1983). \MR{721733}

\bibitem[Roy88]{Royden}
H.~L. Royden, \emph{Real analysis}, third ed., Macmillan Publishing Company,
  New York, 1988. \MR{1013117}

\bibitem[RS75]{RS75}
David Ruelle and Dennis Sullivan, \emph{Currents, flows and diffeomorphisms},
  Topology \textbf{14} (1975), no.~4, 319--327. \MR{415679}

\bibitem[Rue76]{dR76}
David Ruelle, \emph{A measure associated with axiom-{A} attractors}, Amer. J.
  Math. \textbf{98} (1976), no.~3, 619--654. \MR{0415683}

\bibitem[Sin68]{jS68}
Ja.~G. Sina\u{\i}, \emph{Markov partitions and {U}-diffeomorphisms},
  Funkcional. Anal. i Prilo\v{z}en \textbf{2} (1968), no.~1, 64--89.
  \MR{0233038}

\bibitem[VO16]{VO16}
Marcelo Viana and Krerley Oliveira, \emph{Foundations of ergodic theory},
  Cambridge Studies in Advanced Mathematics, vol. 151, Cambridge University
  Press, Cambridge, 2016. \MR{3558990}

\bibitem[Wal75]{pW75}
Peter Walters, \emph{A variational principle for the pressure of continuous
  transformations}, Amer. J. Math. \textbf{97} (1975), no.~4, 937--971.
  \MR{390180}

\end{thebibliography}

\end{document}